\documentclass[11pt]{amsart}
\usepackage{amsthm}
\usepackage{amsmath,amstext,amsthm,amsfonts,amssymb,amsthm}
\usepackage{mathrsfs}
\usepackage{multicol}
\usepackage{latexsym}
\usepackage{color}
\usepackage{graphicx}
\usepackage{epsf}
\usepackage{epsfig}
\usepackage{epic}
\usepackage{graphics}
\usepackage{verbatim}
\usepackage{color}
\usepackage{bm}
\newtheorem{theorem}{Theorem}[section]

\newtheorem{corollary}[theorem]{Corollary}
\theoremstyle{definition}
\newtheorem{definition}[theorem]{Definition}
\newtheorem{proposition}[theorem]{Proposition}
\newtheorem{example}[theorem]{Example}

\theoremstyle{remark}
\newtheorem{remark}[theorem]{Remark}
\numberwithin{equation}{section}

\newcommand{\q}{\mathfrak{q}}
\newcommand{\HI}{\mathfrak{H}}

\newcommand{\C}{\mathbb{C}}

\newcommand{\B}{\mathcal{B}}

\newcommand{\qu}{\mathfrak{q}}
\newcommand{\pu}{\mathfrak{p}}
\newcommand{\oqu}{\overline{\mathfrak{q}}}

\newcommand{\quat}{\mathbb H}
\newcommand{\R}{\Bbb R}

\newcommand{\mc}{\mathcal}

\newcommand{\be}{\begin{equation}}
\newcommand{\en}{\end{equation}}
\newcommand{\D}{{\mc D}}

\newcommand{\N}{\mathbb N}

\newcommand{\bedefin}{\begin{defi}}
	\newcommand{\findefi}{\end{defi} \medskip}
\newcommand{\betheo}{\begin{theorem}$\!\!${\bf \,\,\,}}
	\newcommand{\entheo}{\end{theorem}}
\newcommand{\enth}{\end{theorem}}
\newcommand{\becor}{\begin{cor}$\!\!${\bf .}}
	\newcommand{\encor}{\end{cor}}
\newcommand{\belem}{\begin{lem}$\!\!${\bf .}}
	\newcommand{\enlem}{\end{lem}}
\newcommand{\bea}{\begin{eqnarray}}
\newcommand{\ena}{\end{eqnarray}}
\newcommand{\beano}{\begin{eqnarray*}}
	\newcommand{\enano}{\end{eqnarray*}}
\newcommand{\bee}{\begin{enumerate}}
	\newcommand{\ene}{\end{enumerate}}
\newcommand{\bei}{\begin{itemize}}
	\newcommand{\eni}{\end{itemize}}
\newcommand{\betab}{\begin{tabular}}
	\newcommand{\entab}{\end{tabular}}

\newcommand{\Iop}{{\mathbb{I}_{V_{\mathbb{H}}^{R}}}}

\newcommand{\bfraka}{\mbox{\boldmath $\mathfrak a$}}
\newcommand{\bfrakb}{\mbox{\boldmath $\mathfrak b$}}

\newcommand{\bk}{\mathbf k}

\newcommand{\bi}{\mathbf i}
\newcommand{\bj}{\mathbf j}

\newcommand{\apo}{\sigma_{ap}^S}
\newcommand{\vr}{V_\quat^R}
\newcommand{\ur}{U_\quat^R}
\newcommand{\ra}{\text{ran}}
\newcommand{\kr}{\text{ker}}

\newcommand{\sus}{\sigma_{su}^S}

\newcommand{\lr}{\rho_A^S}
\newcommand{\ls}{\sigma_A^S}
\newcommand{\fa}{\sigma_{\bf{f}}^S}
\newcommand{\sa}{\sigma_S}
\newcommand{\va}{\mathscr{V}_A}
\begin{document}
\title[Decomposable operators in the quaternionic setting]{Decomposable operators, local S-spectrum and S-spectrum in the quaternionic setting}
\author{K. Thirulogasanthar$^{\dagger}$ and B. Muraleetharan$^{\ddagger}$ 
}
\address{$^{\dagger}$ Department of Computer Science and Software Engineering, Concordia University, 1455 De Maisonneuve Blvd. West, Montreal, Quebec, H3G 1M8, Canada.}
\address{$^{\ddagger}$ Department of mathematics and Statistics, University of Jaffna, Thirunelveli, Sri Lanka.}

%\address{$^*$Dipartimento di Matematica, Politecnico di Milano, Via E. Bonardi, 9, 20133, Milano, Italy.}
\email{  santhar@gmail.com  and bbmuraleetharan@jfn.ac.lk 
}
\date{\today}
\date{\today}
\begin{abstract}
In a right quaternionic Hilbert space, following the complex formalism, decomposable operators, the so-called Bishop's property and the single valued extension property are defined and the connections between them are studied to certain extent. In particular, for a decomposable operator, it is shown that the S-spectrum, approximate S-point spectrum, surjectivity S-spectrum and the union of all local S-spectra coincide. Using continuous right slice-regular functions we have also studied certain properties of local S-spectrum and local S-spectral subspaces.
\end{abstract}
\maketitle
\pagestyle{myheadings}
%%%%%%%%%%%%%%%%%%%%%%%%%%%%%%%%%%%%%%%%%%%%%%%%%%%%%%%%%%%%%%%%%%%%%%%%
\section{Introduction}
In complex spectral theory, decomposable operators are closely related to the so-called Bishop's property and the single valued extension property (SVEP).
 The spectrum of a bounded linear operator on a Hilbert space or Banach space can be  divided into several subsets depending on the purpose of the investigation. Further, some of these subsets can also be expressed and analyzed  in terms of the local spectrum at a point of the Hilbert space or Banach space. Bishop's property, SVEP and the local spectral theory are closely linked to vector-valued analytic functions.  For a detail account on the complex theory see \cite{Ai}, \cite{La} and the many references therein.\\

In the complex setting, in a complex Hilbert space or Banach space $\HI$, for a bounded linear operator, $A$, the spectrum is defined as the set of complex numbers $\lambda$ for which the operator  $Q_\lambda(A)=A-\lambda \mathbb{I}_\HI$, where $\mathbb{I}_{\HI}$ is the identity operator on $\HI$,  is not invertible. In the quaternionic setting, let $\vr$ be a separable right quaternionic Hilbert space or Banach space,  $A$ be a bounded right linear operator, and $R_\qu(A)=A^2-2\text{Re}(\qu)A+|\qu|^2\Iop$, with $\qu\in\quat$, the set of all quaternions, be the pseudo-resolvent operator. The S-spectrum is defined as the set of quaternions $\qu$ for which $R_\qu(A)$ is not invertible. In the complex case various classes of spectra, such as approximate point spectrum, surjectivity spectrum etc. are defined by placing restrictions on the operator $Q_\lambda(A)$. In this regard, in the quaternionic setting, these spectra are defined by placing the same restrictions to the operator $R_\qu(A)$.\\

Due to the non-commutativity, in the quaternionic case  there are three types of  Hilbert spaces: left, right, and two-sided, depending on how vectors are multiplied by scalars. This fact can entail several problems. For example, when a Hilbert space $\mathcal H$ is one-sided (either left or right) the set of linear operators acting on it does not have a linear structure. Moreover, in a one sided quaternionic Hilbert space, given a linear operator $A$ and a quaternion $\mathfrak{q}\in\quat$, in general we have that $(\mathfrak{q} A)^{\dagger}\not=\overline{\mathfrak{q}} A^{\dagger}$ (see \cite{Mu} for details). These restrictions can severely prevent the generalization  to the quaternionic case of results valid in the complex setting. Even though most of the linear spaces are one-sided, it is possible to introduce a notion of multiplication on both sides by fixing an arbitrary Hilbert basis of $\mathcal H$.  This fact allows to have a linear structure on the set of linear operators, which is a minimal requirement to develop a full theory. \\

As far as we know, decomposable operators and their S-spectral properties, and their connection to SVEP and the local $S$-spectral theory,  have not been studied in the quaternionic setting yet. In this regard, in this note we investigate these properties in the quaternionic setting. In the complex case, the local spectrum at a point in $\HI$, SVEP, the Bishop's property are defined in terms of vector-valued analytic functions. There have been several attempts to define analyticity in the quaternionic setting by mimicking the complex setting \cite{Am}. However, the most promising, and recent attempt
was the slice-regularity, that is, the slice-regular functions are the quaternionic counterpart of the complex analytic functions \cite{NFC,ghimorper, GSS}. In this regard, we define the local $S$-spectrum, SVEP and Bishop's property in terms of continuous slice-regular functions.\\

Apart from the non-commutativity of quaternions, due to the structure of the operator $R_\qu(A)$ we have experienced severe difficulties in extending several results valid in the complex setting to quaternions. For example, for $\lambda, \mu\in\C$, $Q_\lambda(A)=Q_\mu(A)-(\lambda-\mu)\mathbb{I}_{\HI}$ and this equality plays an important role in proofs of several local spectral results. Unfortunately, a similar equality, in a satisfactory way, could not be obtained for the operator $R_\qu(A)$ by us. Even if we restrict $R_\qu(A)$ to a complex slice within quaternions $Q_\lambda(A)\not=R_\lambda(A)$, therefore, we cannot expect all the results valid in the complex setting to hold for quaternions. \\

The article is organized as follows. In section 2 we introduce the set of quaternions, quaternionic Hilbert spaces and their bases, and slice-regularity as needed for the development of this article, which may not be familiar to a broad range of audience. In section 3 we define and investigate, as needed, right linear operators and their properties and the S-spectrum.  In section 4 we introduce the decomposability of an operator and study the space of continuous right slice-regular functions. In section 5 we recall few results associated with approximate S-point spectrum and surjectivity S-spectrum from \cite{Fr,Ka}. In this section we also prove certain S-spectral inclusions of restriction and quotient of a bounded right linear operator. 
In section 6 we study some connections of decomposability, Bishop's property and SVEP. In section 6.1 we develop some results regarding local $S$-spectrum, local $S$-spectral subspaces. In particular, we show that the S-spectrum, approximate S-point spectrum, surjectivity spectrum and the union of all local S-spectra coincide for a decomposable operator. This phenomena may be used to check the decomposability of certain operators. We also provide an example to validate this claim.

%%%%%%%%%%%%%%%%%%%%%%%%%%%%%%%%%%%%%%%%%%%%%%%%%%%%%%%%%%%%%%%%%%%%%%
\section{Mathematical preliminaries}
In order to make the paper self-contained, we recall some facts about quaternions which may not be well-known.  For details we refer the reader to \cite{Ad,ghimorper,Vis}.
\subsection{Quaternions}
Let $\quat$ denote the field of all quaternions and $\quat^*$ the group (under quaternionic multiplication) of all invertible quaternions. A general quaternion can be written as
$$\qu = q_0 + q_1 \bi + q_2 \bj + q_3 \bk, \qquad q_0 , q_1, q_2, q_3 \in \mathbb R, $$
where $\bi,\bj,\bk$ are the three quaternionic imaginary units, satisfying
$\bi^2 = \bj^2 = \bk^2 = -1$ and $\bi\bj = \bk = -\bj\bi,  \; \bj\bk = \bi = -\bk\bj,
\; \bk\bi = \bj = - \bi\bk$. The quaternionic conjugate of $\qu$ is
$$ \overline{\qu} = q_0 - \bi q_1 - \bj q_2 - \bk q_3 , $$
while $\vert \qu \vert=(\qu \overline{\qu})^{1/2} $ denotes the usual norm of the quaternion $\qu$.
If $\qu$ is non-zero element, it has inverse
$
\qu^{-1} =  \dfrac {\overline{\qu}}{\vert \qu \vert^2 }.$
Finally, the set
\begin{eqnarray*}
\mathbb{S}&=&\{I=x_1 \bi+x_2\bj+x_3\bk~\vert
~x_1,x_2,x_3\in\mathbb{R},~x_1^2+x_2^2+x_3^2=1\},
\end{eqnarray*}
contains all the elements whose square is $-1$. It is a $2$-dimensional sphere in $\mathbb H$ identified with $\mathbb R^4$.
%%%%%%%%%%%%%%%%%%%%%%%%%%%%%%%%%%%%%%%%%%%%%%%%%%%%%%%%%
\subsection{Quaternionic Hilbert spaces}
In this subsection we  discuss right quaternionic Hilbert spaces. For more details we refer the reader to \cite{Ad,ghimorper,Vis}.
\subsubsection{Right quaternionic Hilbert Space}
Let $V_{\quat}^{R}$ be a vector space under right multiplication by quaternions.  For $\phi,\psi,\omega\in V_{\quat}^{R}$ and $\qu\in \quat$, the inner product
$$\langle\cdot\mid\cdot\rangle_{V_{\quat}^{R}}:V_{\quat}^{R}\times V_{\quat}^{R}\longrightarrow \quat$$
satisfies the following properties
\begin{enumerate}
	\item[(i)]
	$\overline{\langle \phi\mid \psi\rangle_{V_{\quat}^{R}}}=\langle \psi\mid \phi\rangle_{V_{\quat}^{R}}$
	\item[(ii)]
	$\|\phi\|^{2}_{V_{\quat}^{R}}=\langle \phi\mid \phi\rangle_{V_{\quat}^{R}}>0$ unless $\phi=0$, a real norm
	\item[(iii)]
	$\langle \phi\mid \psi+\omega\rangle_{V_{\quat}^{R}}=\langle \phi\mid \psi\rangle_{V_{\quat}^{R}}+\langle \phi\mid \omega\rangle_{V_{\quat}^{R}}$
	\item[(iv)]
	$\langle \phi\mid \psi\qu\rangle_{V_{\quat}^{R}}=\langle \phi\mid \psi\rangle_{V_{\quat}^{R}}\qu$
	\item[(v)]
	$\langle \phi\qu\mid \psi\rangle_{V_{\quat}^{R}}=\overline{\qu}\langle \phi\mid \psi\rangle_{V_{\quat}^{R}}$
\end{enumerate}
where $\overline{\qu}$ stands for the quaternionic conjugate. It is always assumed that the
space $V_{\quat}^{R}$ is complete under the norm given above and separable. Then,  together with $\langle\cdot\mid\cdot\rangle$ this defines a right quaternionic Hilbert space. Quaternionic Hilbert spaces share many of the standard properties of complex Hilbert spaces. All the spaces considered in this manuscript is a right quaternionic Hilbert space or Banach space.

The next two Propositions can be established following the proof of their complex counterparts, see e.g. \cite{ghimorper,Vis}.
\begin{proposition}\label{P1}
Let $\mathcal{O}=\{\varphi_{k}\,\mid\,k\in N\}$
be an orthonormal subset of $V_{\quat}^{R}$, where $N$ is a countable index set. Then following conditions are pairwise equivalent:
\begin{itemize}
\item [(a)] The closure of the linear combinations of elements in $\mathcal O$ with coefficients on the right is $V_{\quat}^{R}$.
\item [(b)] For every $\phi,\psi\in V_{\quat}^{R}$, the series $\sum_{k\in N}\langle\phi\mid\varphi_{k}\rangle_{V_{\quat}^{R}}\langle\varphi_{k}\mid\psi\rangle_{V_{\quat}^{R}}$ converges absolutely and it holds:
$$\langle\phi\mid\psi\rangle_{V_{\quat}^{R}}=\sum_{k\in N}\langle\phi\mid\varphi_{k}\rangle_{V_{\quat}^{R}}\langle\varphi_{k}\mid\psi\rangle_{V_{\quat}^{R}}.$$
\item [(c)] For every  $\phi\in V_{\quat}^{R}$, it holds:
$$\|\phi\|^{2}_{V_{\quat}^{R}}=\sum_{k\in N}\mid\langle\varphi_{k}\mid\phi\rangle_{V_{\quat}^{R}}\mid^{2}.$$
\item [(d)] $\mathcal{O}^{\bot}=\{0\}$.
\end{itemize}
\end{proposition}
\begin{definition}
The set $\mathcal{O}$ as in Proposition \ref{P1} is called a {\em Hilbert basis} of $V_{\quat}^{R}$.
\end{definition}
\begin{proposition}\label{P2}
Every quaternionic Hilbert space $V_{\quat}^{R}$ has a Hilbert basis. All the Hilbert bases of $V_{\quat}^{R}$ have the same cardinality.

Furthermore, if $\mathcal{O}$ is a Hilbert basis of $V_{\quat}^{R}$, then every  $\phi\in V_{\quat}^{R}$ can be uniquely decomposed as follows:
$$\phi=\sum_{k\in N}\varphi_{k}\langle\varphi_{k}\mid\phi\rangle_{V_{\quat}^{R}},$$
where the series $\sum_{k\in N}\varphi_k\langle\varphi_{k}\mid\phi\rangle_{V_{\quat}^{R}}$ converges absolutely in $V_{\quat}^{R}$.
\end{proposition}

It should be noted that once a Hilbert basis is fixed, every left (resp. right) quaternionic Hilbert space also becomes a right (resp. left) quaternionic Hilbert space \cite{ghimorper,Vis}.

The field of quaternions $\quat$ itself can be turned into a left quaternionic Hilbert space by defining the inner product $\langle \qu \mid \qu^\prime \rangle = \qu \overline{\qu^{\prime}}$ or into a right quaternionic Hilbert space with  $\langle \qu \mid \qu^\prime \rangle = \overline{\qu}\qu^\prime$.
%%%%%%%%%%%%%%%%%%%%%%%%%%%%%%%%%%%%%%%%%%%%%%%%%%%%%%%%%%%%%%%%
\begin{definition}(Slice-regular functions \cite{Jo})\label{D2}
	Let $\Omega$ be a domain in $\quat$. A real differentiable (i.e., with respect to $x_0$ and the $x_i,\; i=1,2,3$) function $f:\Omega\longrightarrow \vr$ is said to be slice right regular if, for every quaternion $I\in\mathbb{S}$, the restriction of $f$ to the complex plane $\C_I=\mathbb{R}+I\mathbb{R}$ passing through the origin, and containing $1$ and $I$, has continuous partial derivatives (with
	respect to $x$ and $y$, every element in $\C_I$ being uniquely expressible as $x + yI$) and satisfies
	\begin{equation}
		\overline{\partial}_I f (x + yI) := \frac 12\left(\frac {\partial f_I (x + yI )}{\partial x}
		+  \frac {\partial f_I (x + yI )}{\partial y}I\right) = 0\; ,
		\label{rightslicereg}
	\end{equation}
	where $f_I=f|_{\Omega\cap \C_I}$.
\end{definition}

With this definition all monomials of the form $\phi\qu^n,~\phi\in \vr,~n\in\mathbb{N}$, are slice right regular. Since regularity respects addition, all polynomials of the form $f(\qu)=\sum_{t=0}^{n}\phi_t\qu^t$, with $\phi_t\in \vr$, are slice right regular. Further, an analog of Abel's theorem guarantees convergence of appropriate infinite power series.
\begin{proposition}\cite{Gra1}\label{P1}
	For any non-real quaternion $\qu\in \quat\setminus\mathbb{R}$, there exist, and are unique, $x,y\in\mathbb{R}$ with $y>0$, and $I\in\mathbb{S}$ such that $\qu=x+yI$.
\end{proposition}
\begin{definition}\cite{Jo}
Let $f:\Omega\subseteq\quat\longrightarrow\vr$ and $\qu=x+yI\in\Omega.$ If $\qu$ is not real then  we say that $f$ admits right-slice derivative in a non-real point $\qu$ if 
$$\partial_Sf(\qu)=\lim_{\pu\rightarrow\qu, \pu\in L_I}(f_I(\pu)-f_I(\qu))(\pu-\qu)^{-1}$$
exists and finite for any $I\in\mathbb{S}$.\\
\end{definition}
Under the above definition the slice derivative of a regular function is regular. For $\phi_n\in\vr$ we have
\begin{equation}\label{E11}
	\partial_S\left(\sum_{n=0}^{\infty}\phi_n \qu^n\right)=\sum_{n=0}^{\infty} n \phi_n \qu^{n-1}.
\end{equation}
The following theorem gives the quaternionic version of holomorphy via a Taylor series. Let $B(0,R)$ be an open ball in $H$, of radius $R$ and centered at $0$.
\begin{theorem}\cite{Jo, Gra2}\label{T5}
	A function $f: B(0,R)\longrightarrow \vr$ is right regular if and only if it has a series expansion of the form
	$$ f(\qu)=\sum_{n=0}^{\infty}\frac{1}{n!}\frac{\partial^n f}{\partial x^n}(0)\qu^n$$
	converging on $B(0, R)$.
\end{theorem}
%%%%%%%%%%%%%%%%%%%%%%%
For all $\pu,\qu\in\quat$, let
$$\sigma(\qu,\pu)=\left\{\begin{array}{cc}
|\qu-\pu|&\text{if}~~\pu,\qu~~\text{lie in the same complex plane}~~\C_I\\
\omega(\qu,\pu)& \text{otherwise}\end{array}\right.,$$
where $\omega(\qu,\pu)=\sqrt{(\text{Re}(\qu)-\text{Re}(\pu))^2+(|\text{Im}(\qu)|-|\text{Im}(\pu)|)^2}$.
\begin{definition}\label{GD1}\cite{GSS}
	The $\sigma$-ball of radius $R$ centered at $\pu$ is the set
	$$\Sigma(\pu,R)=\{\qu\in\quat~~|~~\sigma(\qu,\pu)<R\}.$$
	Also denote $\Omega(\pu,R)=\{\qu\in\quat~~|~~\omega(\qu,\pu)<R\}$.
	\end{definition}
For right slice regular series the $*_R$ product is defined as
$$\left(\sum_{n=0}^\infty a_n\qu^n\right)*_R\left(\sum_{n=0}^\infty b_n\qu^n\right)=\sum_{n=0}^\infty\left(\sum_{k=0}^n a_kb_{n-k}\right)\qu^n,$$
where $a_n,b_n,\qu\in\quat$. The same product also holds if $a_n,b_n\in\vr$ and $\qu\in\quat$ \cite{Jo}.
%%%%%%%%%%%%%%%%%%%%%%%
\begin{theorem}\label{G2.11}(Theorem 2.11 , \cite{GSS}) Choose any sequence $\{a_n\}_{n\in\N}$ in $\quat$ and let $R\in (0,\infty]$ be such
that $\displaystyle 1/R=\lim \sup_{n\rightarrow\infty}|a_n|^{\frac{1}{n}}$. For all $\pu\in\quat$, the series
$$f(\qu)=\sum_{n=0}^\infty a_n(\qu-\pu)^{*_Rn}$$
converges absolutely and uniformly on the compact subsets of $\Sigma(\pu,R)$, and it does not converge at any point of $\quat\setminus\overline{\Sigma(\pu,R)}$ ( $R$ is called the $\sigma$-radius of convergence
of $f(\qu)$). Furthermore, if $\Omega(\pu,R)\not=\emptyset$, then the sum of the series defines a right regular
function $f:\Omega(\pu,R)\longrightarrow\quat$.
\end{theorem}
%%%%%%%%%%%%%%%%%
\begin{definition}\cite{ESR}
	$U\subseteq\quat$ is called a slice domain if it is a connected set whose intersection with every complex plane $\C_I$ is connected.
\end{definition}
\begin{proposition}\cite{ESR}(Maximum modulus principle) Let $U\subseteq\quat$ be a slice domain let $f:U\longrightarrow\quat$ be a slice regular function. If $|f|$ has a relative maximum at $\pu\in U$, then $f$ is a constant.
\end{proposition}
\begin{proposition}\cite{ESR}(Liouville) Let $f$ be a bounded entire function. Then $f$ is a constant.
\end{proposition}
%%%%%%%%%%%%%%%
\begin{theorem}\label{G2.12}(Theorem 2.12, \cite{GSS}) Let $f$ be a right regular function on a domain $\Omega\subseteq\quat$ and let $\pu\in\Omega$. In
each $\sigma$-ball $\Sigma(\pu,R)$ contained in $\Omega$, the function $f$ expands as
$$f(\qu)=\sum_{n\in\N}\frac{1}{n!}f^{(n)}(\pu)(\qu-\pu)^{*_Rn}.$$
\end{theorem}
Let $U\subseteq\quat$ be an open subset. Denote
$$H(U,\vr)=\{f:U\longrightarrow\vr~~|~~f~~\text{ is continuous and right slice-regular}\}.$$
On $H(U,\vr)$ define $\displaystyle d(f,g)=\sum_{n=1}^{\infty}\frac{1}{2^n}\frac{\|f-g\|_n}{1+\|f-g\|_n}$ for all $f,g\in H(U,\vr)$, where $\|f\|_n=\sup\{\|f(\qu)\|~~|~~\qu\in K_n\}$, and $\{K_n\}_{n\in\N}$ is a sequence of compact subsets of $U$ for which $K_n\subseteq{\text{int}}K_{n+1}$, the interior of $K_{n+1}$, for all $n\in\N$ and $\displaystyle\bigcup_{n\in\N}K_n=U$.
\begin{proposition}\label{H0}
	The space $H(U,\vr)$ is a quaternionic left linear vector space with respect to point-wise vector space operations.
\end{proposition}
\begin{proof}
See remark 4.1.2 in \cite{NFC}.	
\end{proof}
\begin{proposition}\label{H1}
	$d(f,g)$ is a translation invariant metric on $H(U,\vr)$.
\end{proposition}
\begin{proof}
	It is straightforward to check that $d$ is a metric and the equality $d(f+h, g+h)=d(f,g)$ is obvious for all $f,g,h\in H(U,\vr)$.
\end{proof}	
\begin{proposition}\label{Unireg}
Let $\{f_n\}$ be a sequence of right regular functions defined on $U$ and  converging uniformly to a function $f$ on the compact subsets of $U$. Then $f$ is a right regular function.
\end{proposition}
\begin{proof}
We can prove $f$ is right regular by making the same argument of the proof of Proposition 4.4 in \cite{ESR}.
\end{proof}
\begin{proposition}\label{H2}
	$H(U,\vr)$ is a Fr\'echet space with respect to the point-wise vector space operations and the topology of locally uniformly convergence induced by $d(f,g)$.
\end{proposition}
\begin{proof}
The proof is straightforward from the Propositions \ref{H0}, \ref{H1} and \ref{Unireg}. 		
\end{proof}
Proposition \ref{H2} means: if $\{f_n\}\subseteq H(U,\vr)$ such that $f_n\longrightarrow 0$ as $n\rightarrow\infty$ in the topology of $H(U,\vr)$ precisely when $\{f_n\}$ converges uniformly to zero on each compact subset of $U$.

%%%%%%%%%%%%%%%%%%%%%%%%%%%%%%%%%%%%%%%%%%%%%%%%%%%%%%%%%%%%%%%%%%%%%%%%%%%%%%%%%%
\section{Right quaternionic linear  operators and some basic properties}
In this section we shall define right  $\quat$-linear operators and recall some basis properties as needed for the development of this manuscript. Most of them are very well known. In this manuscript, we follow the notations in \cite{AC} and \cite{ghimorper}. 
\begin{definition}
A mapping $A:\D(A)\subseteq V_{\quat}^R \longrightarrow U_{\quat}^R$, where $\D(A)$ stands for the domain of $A$, is said to be right $\quat$-linear operator or, for simplicity, right linear operator, if
$$A(\phi\bfraka+\psi\bfrakb)=(A\phi)\bfraka+(A\psi)\bfrakb,~~\mbox{~if~}~~\phi,\,\psi\in \D(A)~~\mbox{~and~}~~\bfraka,\bfrakb\in\quat.$$
\end{definition}
The set of all right linear operators from $V_{\quat}^{R}$ to $U_{\quat}^{R}$ will be denoted by $\mathcal{L}(V_{\quat}^{R},U_{\quat}^{R})$ and the identity linear operator on $V_{\quat}^{R}$ will be denoted by $\Iop$. For a given $A\in \mathcal{L}(V_{\quat}^{R},U_{\quat}^{R})$, the range and the kernel will be
\begin{eqnarray*}
\text{ran}(A)&=&\{\psi \in U_{\quat}^{R}~|~A\phi =\psi \quad\text{for}~~\phi \in\D(A)\}\\
\ker(A)&=&\{\phi \in\D(A)~|~A\phi =0\}.
\end{eqnarray*}
We call an operator $A\in \mathcal{L}(V_{\quat}^{R},U_{\quat}^{R})$ bounded if
\begin{equation}\label{PE1}
\|A\|=\sup_{\|\phi \|_{\vr}=1}\|A\phi \|_{\ur}<\infty,
\end{equation}
or equivalently, there exist $K\geq 0$ such that $\|A\phi \|_{\ur}\leq K\|\phi \|_{\vr}$ for all $\phi \in\D(A)$. The set of all bounded right linear operators from $V_{\quat}^{R}$ to $U_{\quat}^{R}$ will be denoted by $\B(V_{\quat}^{R},U_{\quat}^{R})$. Set of all  invertible bounded right linear operators from $V_{\quat}^{R}$ to $U_{\quat}^{R}$ will be denoted by $\mathcal{G} (V_{\quat}^{R},U_{\quat}^{R})$. We also denote for a set $\Delta\subseteq\quat$, $\Delta^*=\{\oqu~|~\qu\in\Delta\}$.
\\
Assume that $V_{\quat}^{R}$ is a right quaternionic Hilbert space, $A$ is a right linear operator acting on it.
Then, there exists a unique linear operator $A^{\dagger}$ such that
\begin{equation}\label{Ad1}
\langle \psi \mid A\phi \rangle_{\ur}=\langle A^{\dagger} \psi \mid\phi \rangle_{\vr};\quad\text{for all}~~~\phi \in \D (A), \psi\in\D(A^\dagger),
\end{equation}
where the domain $\D(A^\dagger)$ of $A^\dagger$ is defined by
$$
\D(A^\dagger)=\{\psi\in U_{\quat}^{R}\ |\ \exists \varphi\ {\rm such\ that\ } \langle \psi \mid A\phi \rangle_{\ur}=\langle \varphi \mid\phi \rangle_{\vr}\}.$$
The following theorem gives two important and fundamental results about right $\quat$-linear bounded operators which are already appeared in \cite{ghimorper} for the case of $\vr=\ur$. Point (b) of the following theorem is known as the open mapping theorem.
\begin{theorem}\cite{Fr}\label{open} Let $A:\D(A)\subseteq V_{\quat}^R \longrightarrow U_{\quat}^R$ be a right $\quat$-linear operator. Then%\cite{ghimorper}
\begin{itemize} 
\item[(a)] $A\in\B(\vr,\ur)$ if and only if $A$ is continuous.
\item[(b)] if $A\in\B(\vr,\ur)$ is surjective, then $A$ is open. In particular, if $A$ is bijective then $A^{-1}\in\B(\vr,\ur)$.
\end{itemize}
\end{theorem}
The following proposition provides some useful aspects about the orthogonal complement subsets.
\begin{proposition}\cite{Fr}\label{ort}
Let $M\subseteq\vr$. Then
\begin{itemize}
\item [(a)] $M^{^\perp}$ is closed.
\item [(b)] if $M$ is a closed subspace of $\vr$ then $\vr=M\oplus M^\perp$.
\item [(c)] if $\dim(M)<\infty$, then $M$ is a closed subspace.
\end{itemize}
\end{proposition}
\begin{proposition} \cite{ghimorper, Fr} \label{IP30}
Let $A\in\B(\vr, \ur)$. Then
\begin{itemize}
\item [(a)] $\ra(A)^\perp=\kr(A^\dagger).$
\item [(b)] $\kr(A)=\ra(A^\dagger)^\perp.$
\item [(c)] $\kr(A)$ is closed subspace of $\vr$.
\end{itemize}
\end{proposition}
\begin{proposition}\label{dag}\cite{ghimorper} $A\in\B(\vr)$, then $A^\dagger\in\B(\vr), \|A\|=\|A^\dagger\|$ and $\|A^\dagger A\|=\|A\|^2$.
\end{proposition}
\begin{definition}\cite{Ai}\label{BBD}
An operator $A\in\B(\vr)$ is said to be bounded below if $A$ is injective and has closed range.
\end{definition}
\begin{proposition}\label{BBP}
$A\in\B(\vr)$ is bounded below if and only if there exists $K>0$ such that $\|A\phi\geq K\|\phi\|$ for all $\phi\in\vr$.
\end{proposition}
\begin{proof}
	A proof follows exactly as a complex proof. For a complex proof see \cite{Ai}, page 15.	
\end{proof}
\begin{theorem}\cite{Fr}(Bounded inverse theorem)\label{NT1}
Let $A\in\B(\vr,\ur)$, then the following results are equivalent.
\begin{enumerate}
\item [(a)] $A$ has a bounded inverse on its range.
\item[(b)] $A$ is bounded below.
\item[(c)] $A$ is injective and has a closed range.
\end{enumerate}
\end{theorem}
\begin{proposition}\cite{Fr}\label{NP1}
Let $A\in\B(\vr,\ur)$, then $\ra(A)$ is closed in $\ur$ if and only if $\ra(A^\dagger)$ is closed in $\vr$.
\end{proposition}

\begin{proposition}\cite{Fr}\label{NP2}
Let $A\in\B(\vr)$. Then,
 $A$ is invertible if and only if it is injective with a closed range (i.e., $\kr(A)=\{0\}$ and $\overline{\ra(A)}=\ra(A)$).
\end{proposition}

\begin{definition}\cite{La}\label{ID1}
Let $A\in\B(\vr)$. A closed subspace $M\subseteq\vr$ is said to be $A$-invariant  if $A(M)\subseteq M$, where $A(M)=\{A\phi~|~\phi\in M\}$. It is said to be $A$-hyperinvariant if $B(M)\subseteq M$ for every $B\in\B(\vr)$ that commutes with $A$.
\end{definition}

%%%%%%%%%%%%%%%%%%%%%%%%%%%%%%%%%%%%%%%%%%%%%%%%%%
%%%%%%%%%%%%%%%%%%%%%%%%%%%%%%%%%%%%%%%%%%%%%%%%%%
\subsection{S-Spectrum}
For a given right linear operator $A:\D(A)\subseteq V_{\quat}^R\longrightarrow V_{\quat}^R$ and $\qu\in\quat$, we define the operator $R_{\qu}(A):\D(A^{2})\longrightarrow\quat$ by  $$R_{\qu}(A)=A^{2}-2\text{Re}(\qu)A+|\qu|^{2}\Iop,$$
where $\qu=q_{0}+\bi q_1 + \bj q_2 + \bk q_3$ is a quaternion, $\text{Re}(\qu)=q_{0}$  and $|\qu|^{2}=q_{0}^{2}+q_{1}^{2}+q_{2}^{2}+q_{3}^{2}.$\\
In the literature, the operator is called pseudo-resolvent since it is not the resolvent operator of $A$ but it is the one related to the notion of spectrum as we shall see in the next definition. For more information, on the notion of $S$-spectrum the reader may consult e.g. \cite{Al,Fab, Fab1, NFC}, and  \cite{ghimorper}.
\begin{definition}
Let $A:\D(A)\subseteq V_{\quat}^R\longrightarrow V_{\quat}^R$ be a right linear operator. The {\em $S$-resolvent set} (also called \textit{spherical resolvent} set) of $A$ is the set $\rho_{S}(A)\,(\subset\quat)$ such that the three following conditions hold true:
\begin{itemize}
\item[(a)] $\ker(R_{\qu}(A))=\{0\}$.
\item[(b)] $\text{ran}(R_{\qu}(A))$ is dense in $V_{\quat}^{R}$.
\item[(c)] $R_{\qu}(A)^{-1}:\text{ran}(R_{\qu}(A))\longrightarrow\D(A^{2})$ is bounded.
\end{itemize}
The \textit{$S$-spectrum} (also called \textit{spherical spectrum}) $\sigma_{S}(A)$ of $A$ is defined by setting $\sigma_{S}(A):=\quat\smallsetminus\rho_{S}(A)$. For a bounded linear operator $A$ we can write the resolvent set as
\begin{eqnarray*}
\rho_S(A)&=& \{\qu\in\quat~|~R_\qu(A)\in\mathcal{G}(V_{\quat}^R)\}\\
&=&\{\qu\in\quat~|~R_\qu(A)~\text{has an inverse in}~\B(V_{\quat}^R)\}\\
&=&\{\qu\in\quat~|~\text{ker}(R_\qu(A))=\{0\}\quad\text{and}\quad \text{ran}(R_\qu(A))=V_\quat^R\}
\end{eqnarray*}
and the spectrum can be written as
\begin{eqnarray*}
\sigma_S(A)&=&\quat\setminus\rho_S(A)\\
&=&\{\qu\in\quat~|~R_\qu(A)~\text{has no inverse in}~\B(V_{\quat}^R)\}\\
&=&\{\qu\in\quat~|~\text{ker}(R_\qu(A))\not=\{0\}\quad\text{or}\quad \text{ran}(R_\qu(A))\not=V_\quat^R\}
\end{eqnarray*}
The spectrum $\sigma_S(A)$ decomposes into three major disjoint subsets as follows:
\begin{itemize}
\item[(i)] the \textit{spherical point spectrum} of $A$: $$\sigma_{pS}(A):=\{\qu\in\quat~\mid~\ker(R_{\qu}(A))\ne\{0\}\}.$$
\item[(ii)] the \textit{spherical residual spectrum} of $A$: $$\sigma_{rS}(A):=\{\qu\in\quat~\mid~\ker(R_{\qu}(A))=\{0\},\overline{\text{ran}(R_{\qu}(A))}\ne V_{\quat}^{R}~\}.$$
\item[(iii)] the \textit{spherical continuous spectrum} of $A$: $$\sigma_{cS}(A):=\{\qu\in\quat~\mid~\ker(R_{\qu}(A))=\{0\},\overline{\text{ran}(R_{\qu}(A))}= V_{\quat}^{R}, R_{\qu}(A)^{-1}\notin\B(V_{\quat}^{R}) ~\}.$$
\end{itemize}
If $A\phi=\phi\qu$ for some $\qu\in\quat$ and $\phi\in V_{\quat}^{R}\smallsetminus\{0\}$, then $\phi$ is called an \textit{eigenvector of $A$ with right eigenvalue} $\qu$. The set of right eigenvalues coincides with the point $S$-spectrum, see \cite{ghimorper}, proposition 4.5.
\end{definition}
\begin{proposition}\cite{Fab2, ghimorper}\label{PP1}
For $A\in\B(\vr)$, the resolvent set $\rho_S(A)$ is a non-empty open set and the spectrum $\sigma_S(A)$ is a non-empty compact set.
\end{proposition}
\begin{remark}\label{R1}
For $A\in\B(\vr)$, since $\sigma_S(A)$ is a non-empty compact set so is its boundary. That is, $\partial\sigma_S(A)=\partial\rho_S(A)\not=\emptyset$.
\end{remark}
%%%%%%%%%%%%%%%%%%%%%%%%%%%%%%%%%%%
\section{Decomposable operators in $\vr$}
For $A\in\B(\vr)$ and an $A$-invariant subspace $Y$ of $\vr$, the operator $A|Y\in\B(Y)$ denotes the operator given by the restriction of $A$ to $Y$ and the operator $A/Y\in\B(\vr/Y)$ denotes the operator induced by $A$ on the quotient space $\vr/Y$. The following definition is an adaptation of the complex definition given in \cite{La}.
\begin{definition}
	An operator $A\in\B(\vr)$ is decomposable if every open cover $\quat=U\cup V$ by two open sets $U$ and $V$ effect the splitting of the spectrum $\sigma_S(A)$ and of the space $\vr$, in the sense that there exist $A$-invariant closed right linear subspaces $Y$ and $Z$ of $\vr$ for which $\sigma_S(A|Y)\subseteq U$, $\sigma_S(A|Z)\subseteq V$, and $\vr=Y+Z$.
\end{definition}
In the above definition
\begin{enumerate}
	\item [(a)]the sum decomposition is, in general, not direct;
	\item[(b)] the spectra of the restrictions not necessarily disjoint.
\end{enumerate}

\begin{proposition}\label{H3} Let $A\in\B(\vr)$ and $U\subseteq\quat$ is an open subset. Then the operator $A^\#:H(U,\vr)\longrightarrow H(U,\vr)$ defined by the composition $(A^\#f)(\qu)=A(f(\qu))$ is a continuous linear operator.
	\end{proposition}
\begin{proof}
	Since $A$ and $f$ are continuous $A^\#$ is continuous and the linearity of $A^\#$ follows from the linearity of $A$.
\end{proof}
There exists an open ball with center $\qu_0$ and radius $r>0$ such that for each $f\in H(U,\vr)$ there is a power series expansion of the form
\begin{equation}\label{DE1}
f(\qu)=\sum_{n=0}^\infty\phi_n(\qu-\qu_0)^{*_Rn}\quad\text{for all}~~\qu\in U,
\end{equation}
where $\{\phi_n\}\subseteq\vr$. The convergence of the series on $U$ is locally uniform. Therefore, the continuity of $A^\#$ implies that
\begin{equation}\label{DE2}
(A^\#f)(\qu)=\sum_{n=0}^\infty A(\phi_n)(\qu-\qu_0)^{*_Rn}\quad\text{for all}~~\qu\in U.
\end{equation}
%%%%%%%%%%%%
\begin{proposition}\label{1.2.1}
	Let $A\in\B(\vr)$ be a continuous linear surjection. Then, for every open ball $U\subseteq\quat$, the induced composition operator $A^{\#}:H(U,\vr)\longrightarrow H(U,\vr)$ is a continuous and open surjection.
\end{proposition}
\begin{proof}
	Let $U=B_\quat(\qu_0,r)$, where $\qu_0\in\quat$ and $r>0$ (guaranteeing the convergence). Then for every function $g\in H(U,\vr)$ we have a power series expansion
	$$g(\qu)=\sum_{n=0}^{\infty}\phi_n(\qu-\qu_0)^{*_Rn}$$
	holding for every $\qu\in U$. The radius of convergence of this power series is $r$. Next observe that, by the open mapping theorem, there exists a constant $c>0$ such that for every $\psi\in\vr$, there is a $\phi\in\vr$ such that $A\phi=\psi$ and $\|\phi\|\leq c\|\psi\|$. In particular, for every $n\in\N_0=\{0,1,2,\cdots\}$, we may choose $\psi_n\in\vr$ such that $A\psi_n=\phi_n$ and $\|\psi_n\|\leq c\|\phi_n\|$. Because
	$$\lim\sup_{n\rightarrow\infty}\|\psi_n\|^{1/n}\leq\lim\sup_{n\rightarrow\infty}c^{1/n}\|\phi_n\|^{1/n}=\lim\sup_{n\rightarrow\infty}\|\phi_n\|^{1/n}\leq\frac{1}{r}.$$
	We conclude that the power series 
		$$f(\qu)=\sum_{n=0}^{\infty}\psi_n(\qu-\qu_0)^{*_Rn}$$
		converges for all $\qu\in U$ and hence defines a function $f\in H(U,\vr)$. Since
		$$(A^{\#}f)(\qu)=\sum_{n=0}^{\infty}A(\psi_n)(\qu-\qu_0)^{*_Rn}=\sum_{n=0}^{\infty}\phi_n(\qu-\qu_0)^{*_Rn}=g(\qu),$$
		$A^{\#}$ is surjective, and hence by the open mapping theorem $A^{\#}$ is open.
\end{proof}
\begin{proposition}\label{1.2.2}
	Let $Y, Z$ be right linear subspaces of $\vr$, and let $U\subseteq\quat$ be an open ball. Then the following properties hold.
	\begin{enumerate}
		\item [(a)] The quotient mapping from $\vr$ to $\vr/Y$ induces a canonical topological linear isomorphism $H(U,\vr)\cong H(U,\vr)/H(U,Y).$
		\item[(b)] If $\vr=Y+Z$, then the operator $\Phi:H(U,Y)\times H(U,Z)\longrightarrow H(U,\vr)$ given by $\Phi(f,g)(\qu)=f(\qu)+g(\qu)$ for all $(f,g)\in H(U,Y)\times H(U,Z)$ and $\qu\in U$ is a continuous and open linear surjection.
		\item[(c)] If $\vr=Y\oplus Z$ holds as a direct sum, then the mapping in part (b) yields a canonical identification $H(U,\vr)\cong\ H(U,Y)\oplus H(U,Z)$.
\end{enumerate}
\end{proposition}
\begin{proof}
	(a)~~If $Q:\vr\longrightarrow\vr/Y$ denote the canonical quotient mapping, then $\kr(Q^{\#})=H(U,Y)$. Hence the assertion is immediate from proposition \ref{1.2.1}.\\
	(b)~~First note that the definition $\Psi(f,g)(\qu):=(f(\qu),g(\qu))$ for all $(f,g)\in H(U,Y)\times H(U,Z)$ and $\qu\in U$ yields a topological linear isomorphism $\Psi$ from the product $\in H(U,Y)\times H(U,Z)$ onto $H(U,Y\times Z)$. Now let $B:Y\times Z\longrightarrow\vr$ be the canonical continuous linear surjection from $Y\times Z$ onto $\vr$. By proposition \ref{1.2.1}, the corresponding mapping $B^{\#}:H(U,Y\times Z)\longrightarrow H(U,\vr)$ is a continuous open linear surjection. Since $\Phi=B^{\#}\circ\Psi$, $\Phi$ is a continuous and open linear surjection.\\
	(c)~~The result is included in part (b).
\end{proof}
%%%%%%%%%%%%
\begin{definition}\label{full}
	Let $\fa(A)$ denote the full S-spectrum of $A\in\B(\vr)$, that is, $\fa(A)$ is the union of the S-spectrum $\sigma_S(A)$ and all bounded connected components of the resolvent set $\rho_S(A)$. Geometrically, $\fa(A)$ is obtained from $\sigma_S(A)$ by filling all the holes of $\sigma_S(A)$.
\end{definition}
%%%%%%%%%%%%

%%%%%%%%%%%%%%%%%%%%%%%%%%%%%%%%%%%%%%%%%%%%%%%%%%%%%%%%%%%%%%%%%
%%%%%%%%%%%%%%%%%%%%%%%%%%%%%%%%%%%%%%%%%%%%%%%
\section{Surjectivity $S$-spectrum and Approximate $S$-point spectrum}
We recall some results regarding the approximate spherical point spectrum and surjectivity S-spectrum of $A\in\B(V_\quat^R)$ from \cite{Fr,Ka} as needed.
\begin{definition}\cite{Fr}\label{D1}
Let $A\in\B(V_\quat^R)$. The {\em approximate S-point spectrum} of $A$, denoted by $\apo(A)$, is defined as
$$\apo(A)=\{\qu\in\quat~~|~~\text{there is a sequence}~~\{\phi_n\}_{n=1}^{\infty}~~\text{such that}~~\|\phi_n\|=1~~\text{and}~~\|R_\qu(A)\phi_n\|\longrightarrow 0\}.$$
\end{definition}
\begin{proposition}\label{P3}\cite{Fr}
Let $A\in\B(\vr)$, then $\sigma_{pS}(A)\subseteq\apo(A)$. 
\end{proposition}
\begin{proposition}\label{P4}\cite{Fr}
If $A\in\B(\vr)$ and $\qu\in\quat$, then the following statements are equivalent.
\begin{enumerate}
\item[(a)] $\qu\not\in\apo(A).$
\item[(b)] $\text{ker}(R_\qu(A))=\{0\}$ and $\text{ran}(R_\qu(A))$ is closed.
\item[(c)] There exists a constant $c\in\R$, $c>0$ such that $\|R_\qu(A)\phi\|\geq c\|\phi\|$ for all $\phi\in\D(A^2)$.
\end{enumerate}
\end{proposition}
\begin{theorem}\label{T1}\cite{Fr}Let $A\in\B(\vr)$, then $\apo(A)$ is a non-empty closed subset of $\quat$ and $\partial\sigma_S(A)\subseteq\apo(A),$ where $\partial\sigma_S(A)$ is the boundary of $\sigma_S(A)$.
\end{theorem}

%%%%%%%%%%%

\begin{definition}\label{DC}
The spherical compression spectrum of an operator $A\in\B(\vr)$, denoted by $\sigma_c^S(A)$, is defined as
$$\sigma_c^S(A)=\{\qu\in\quat~~|~~\text{ran}(R_\qu(A))~~~\text{is not dense in}~~\vr~\}.$$
\end{definition}
\begin{proposition}\label{CP1}\cite{Fr}
Let $A\in\B(\vr)$ and $\qu\in\quat$. Then,
\begin{enumerate}
\item[(a)] $\qu\in\sigma_c^S(A)$ if and only if $\oqu\in\sigma_{pS}(A)$.
\item[(b)] $\sigma_c^S(A)\subseteq\sigma_r^S(A)$.
\item[(c)] $\sigma_S(A)=\apo(A)\cup\sigma_c^S(A)$.
\end{enumerate}
\end{proposition}
\begin{definition}\label{su} Let $A\in\B(\vr)$. The surjectivity S-spectrum of $A$ is defined as
	$$\sus(A)=\{\qu\in\quat~|~\text{ran}(R_{\qu}(A)\not=\vr\}.$$
\end{definition}
Clearly we have
\begin{equation}\label{sue1}
\sigma_c^S(A)\subseteq\sus(A)\quad\text{and}\quad\sigma_S(A)=\sigma_{pS}(A)\cup\sus(A).
\end{equation}
%%%%%%%%%
\begin{proposition}\label{su1}\cite{Ka} Let $A\in\B(\vr)$. Then $A$ has the following properties.
	\begin{enumerate}
		\item[(a)]$\sigma_{pS}(A)\subseteq\sigma_c^S(A^\dagger)~~\text{and}~~\sigma_c^S(A)=\sigma_{pS}(A^\dagger)$.
		\item[(b)]$\sus(A)=\apo(A^\dagger)~~\text{and}~~\apo(A)=\sus(A^\dagger).$
		\item[(c)]$\sigma_S(A)=\sigma_S(A^\dagger).$
	\end{enumerate}
\end{proposition}

%%%%%%%%%%%%%%%%%%%%%%%%%%%%%
\begin{proposition}\label{su2}\cite{Ka}
For $A\in\B(\vr)$, $\sus(A)$ is closed and $\partial\sigma_S(A)\subseteq\sus(A).$
\end{proposition}
%%%%%%%%%%%%%%%%%%
\begin{proposition}\label{1.2.4}
	Let $A\in\B(\vr)$ and suppose that $Y$ and $Z$ be $A$-invariant closed right linear subspaces of $\vr$ with the property that $\vr=Y+Z$. Then
	\begin{enumerate}
		\item [(a)] $\sigma_S(A/Y)\subseteq\sigma_S(A)\cup\sigma_S(A|Y)\subseteq\fa(A)$;
		\item[(b)]$\sa(A/Z)\subseteq\fa(A|Y)\subseteq\fa(A).$
	\end{enumerate}
\end{proposition}
\begin{proof}
	(a)~~We have $A/Y:\vr/Y\longrightarrow\vr/Y$ and $Q:\vr\longrightarrow\vr/Y$, $Q$ is the natural quotient mapping. For an arbitrary $\phi\in\vr$ let $A(\phi)=\psi$. Then $(A/Y)Q(\phi)=(A/Y)(\phi+Y)=\psi+Y$ and $QA(\phi)=Q(A(\phi))=Q(\psi)=\psi+Y$. Thus
	\begin{equation}\label{FE1} 
	(A/Y)Q=QA.
	\end{equation}
	Let $\qu\in\rho_S(A)\cap\rho_S(A|Y)$, then $R_\qu(A)$ is surjective. Since $Q$ is also surjective, by equation \ref{FE1}, $(A/Y)$ is surjective. Since $\qu\in\rho_S(A|Y)$, $\kr(A|Y)=\{0\}$, if $\phi\in\vr$ satisfies $R_\qu(A/Y)Q\phi=0$, then $QR_\qu(A)\phi=0$, and hence $R_\qu(A)\phi\in Y$. Therefore, since $\kr(A)=\{0\}$ and $\kr(A/Y)=\kr(A)+Y=Y$, we have $\phi\in Y$. Thus $R_\qu(A/Y)$ is invertible, and hence $\qu\in\rho_S(A/Y)$. Therefore.
	$$\sa(A/Y)\subseteq\sa(A)\cup\sa(A|Y).$$
	From proposition \ref{T1} we have
	$$\partial\sa(A|Y)\subseteq\apo(A|Y)\subseteq\apo(A)\subseteq\sa(A)$$
	and thus $\fa(A|Y)\subseteq\fa(A)$, and from which we get $\sa(A)\cup\sa(A|Y)\subseteq\fa(A)$. The assertion (a) is established.\\
	(b)~~Since $\vr=Y+Z$ we have a canonical surjection $S:Y\longrightarrow\vr/Z$ by $S(\phi)=\phi+Z$ with $\kr(S)=Y\cap Z$. Let $R:Y/(Y\cap Z)\longrightarrow\vr/Z$ by $\phi+(Y\cap Z)\mapsto \phi+Z$ denote the corresponding isomorphism. Then, for $\phi\in Y$, we have
	$(A/Z)R(\phi+(Y\cap Z))=(A/Z)(\phi+Z)=A(\phi)+Z$ and $R(A|Y)/(Y\cap Z)(\phi+(Y\cap Z))=R(A(\phi)+(Y\cap Z))=A(\phi)+Z$, and thus 
	\begin{equation}\label{FE2}
	(A/Z)R=R(A|Y)/(Y\cap Z).
	\end{equation}
	For simplicity, denote $C=(A/Z)$ and $D=(A|Y)/(Y\cap Z)$. Then equation \ref{FE2} reads $CR=RD$. Since $R$ is invertible, $R^{-1}CR=D$, and thus $R_\qu(R^{-1}CR)=R_\qu(B)$. That is,
	$$R^{-1}C^2R-2\text{Re}(\qu)R^{-1}CR+|\qu|^2\mathbb{I}_{Y/(Y\cap Z)}=R_\qu(D).$$ From this, as $\text{Re}(\qu)$ and $|\qu|$ are real, we get $R^{-1}R_\qu(C)R=R_\qu(D)$. Therefore, $R_\qu(C)$ is invertible if and only if $R_\qu(D)$ is invertible. Hence, we get $$\sa(A/Y)=\sa((A|Y)/(Y\cap Z)).$$
	Now from part (a) we have
	$$\sa((T|Y)/(Y\cap Z))\subseteq\sa(A|Y)\cup\sa(A|(Y\cup Z))\subseteq\fa(A|Y)\subseteq\fa(A).$$
	Therefore we have $\sa(A/Z)\subseteq\fa(A|Y)\subseteq\fa(A)$.
\end{proof}
The following corollary provides a more symmetric picture of spectral inclusions.
%%%%%%%%%%%%%%%
\begin{corollary}\label{NC1}
	Let $A\in\B(\vr)$, $Y$ be an $A$-invariant closed right linear subspace of $\vr$. Then
	\begin{enumerate}
		\item [(a)] $\sa(A|Y)\subseteq\sa(A)\cup\sa(A/Y)$;
		\item[(b)] $\sa(A)\subseteq\sa(A|Y)\cup\sa(A/Y).$
	\end{enumerate}
\end{corollary}
\begin{proof}
	(a)~~Since, by theorem \ref{T1}, $\partial\sa(A|Y)\subseteq\apo(A|Y)\subseteq\apo(A)\subseteq\sa(A)$
	we have $\sa(A|Y)\subseteq\sa(A).$\\
	(b)~~ if $\qu\not\in\sa(A|Y)\cup\sa(A/Y)$, 
	then $R_\qu(A|Y)$ and 
	$R_\qu(A/Y)$ are invertible. 
	Further, if $R_\qu(A)$ is not onto then $R_\qu(A/Y)$ is not onto, 
	and also if $R_\qu(A)$ is not injective then $R_\qu(A/Y)$ is not injective. Therefore, $R_\qu(A)$ is invertible, and hence $\qu\not\in\sa(A)$.
\end{proof}
%%%%%%%%%%%%%%%%%%%%%%%%%%%%%%%%%%%%
\section{decomposability, Bishop's property and SVEP in $\vr$}
In the complex theory, the so-called Bishop's property, called property ($\beta$), plays a central role in local spectral theory. Further, decomposability, property ($\beta$) and the single valued extension property, abbreviated SVEP are closely related to each other \cite{Ai, La}. Following the complex formalism, we examine these properties in $\vr$.
%%%%%%%
\begin{definition}\label{1.2.5}
	An operator $A\in\B(\vr)$ has Bishop's property ($\beta$) if, for every open subset $U$ of $\quat$ and every sequence of continuous right slice-regular functions $f_n:U\longrightarrow\vr$ with the property that $R_\qu(A)f_n(\qu)\longrightarrow 0$ as $n\rightarrow\infty$ uniformly on all compact subsets of $U$, it follows that $f_n(\qu)\longrightarrow 0$ as $n\rightarrow\infty$ again locally uniformly on $U$.
	\end{definition}
For every $A\in\B(\vr)$ and every open set $U\subseteq\quat$, define the operator $A_U:H(U,\vr)\longrightarrow H(U,\vr)$ by $(A_Uf)(\qu)=R_\qu(A)f(\qu)$ for all $f\in H(U,\vr)$ and $\qu\in U$. Then clearly $A_U$ is a continuous linear operator on $H(U,\vr)$.
%%%%%%
\begin{proposition}\label{1.2.6}
	An operator $A\in\B(\vr)$ has property ($\beta$) if and only if, for every open set $U\subseteq\quat$, the operator $A_U$ on $H(U,\vr)$ is injective and has closed range.
\end{proposition}
\begin{proof}
	($\Leftarrow$) Suppose that $A_U$ is injective and has closed range for all open sets $U\subseteq\quat$.\\
	Then, given an open set $U\subseteq\quat$, the open mapping theorem ensures that the operator $A_U$ has a continuous inverse on its range, $\ra(A_U)$. We call this inverse $B_U$. Hence, if $A_Uf_n\longrightarrow 0$ as $n\rightarrow\infty$ in the topology of $H(U,\vr)$, then, clearly, $f_n=B_UA_Uf_n\longrightarrow 0$ as $n\rightarrow\infty$, again in the topology of $H(U,\vr)$. Since $U$ is an arbitrary open set in $\quat$, this proves that $A$ has property ($\beta$).\\
	($\Rightarrow$) Suppose that $A$ has property ($\beta$). Then an obvious consideration of constant sequence in $H(U,\vr)$ shows that $A_U$ is injective for any open set $U\subseteq\quat$.\\
	{\em Claim:} $\ra(A_U)$ is closed.\\
	Let $g\in\overline{\ra(A_U)}$. Then there exists a sequence $\{f_n\}_{n\in\N}\subseteq H(U,\vr)$ such that $A_Uf_n\longrightarrow g$ as $n\rightarrow\infty$ in $H(U,\vr)$. Then $\{f_n\}_{n\in\N}$ is a Cauchy sequence in the metric of $H(U,\vr)$, because if it is not a Cauchy sequence, then we can construct a subsequence $\{f_{n(k)}\}_{k\in\N}$ of  $\{f_n\}_{n\in\N}$ for which the sequence given by $h_k=f_{n(k+1)}-f_{n(k)}$ for all $k\in\N$ did not converge to zero in the metric of $H(U,\vr)$, where as obviously $A_Uh_k\longrightarrow 0$ as $k\rightarrow\infty$, and therefore, by property ($\beta$),  $h_k\longrightarrow 0$ as $k\rightarrow\infty$ in $H(U,\vr)$, which is a contradiction. Since $H(U,\vr)$ is a Fr\'echet space, it follows that there exists an element $f\in H(U,\vr)$ such that $f_n\longrightarrow f$ as $n\rightarrow\infty$ in $H(U,\vr)$. By the continuity of $A_U$ we get $A_uf_n\longrightarrow A_Uf$ as $n\rightarrow\infty$, and hence $g=A_Uf\in\ra(A_U)$. Therefore $A_U$ has closed range for every $U\subseteq\quat$.
\end{proof}
%%%%%%%%%%%%%%%%%%%
\begin{theorem}\label{1.2.7} Every decomposable operator $A\in\B(\vr)$ has property ($\beta$).
\end{theorem}
\begin{proof}
Suppose that $A\in\B(\vr)$ is decomposable. Let $U\subseteq\quat$ be an open set, and consider a sequence of continuous right slice-regular functions
\begin{equation}\label{DE3}
 f_n:U\longrightarrow\vr~~\text{ for which}~~ R_\qu(A)f_n(\qu)\longrightarrow 0~~\text{ as}~~ n\rightarrow\infty
\end{equation}
 locally uniformly on $U$. To show $A$ has property ($\beta$), it is enough to prove that $f_n\longrightarrow 0$, as $n\rightarrow\infty$, uniformly on any closed ball contained in $U$. Given an arbitrary closed ball $D\subseteq U$, we choose an open ball $E$ for which $D\subseteq E\subseteq\overline{E}\subseteq U$. Now we apply the definition of decomposability of $A$ to the open cover $\{E,\quat\setminus D\}$ of $\quat$. This gives us $A$-invariant closed right linear subspaces $Y,Z\subseteq\vr$ for which $\sa(A|Y)\subseteq E$, $\sa(A|Z)\cap D=\emptyset$ and $\vr=Y+Z$. Then, by proposition \ref{1.2.2}, part (c), we obtain sequences $\{g_n\}_{n\in\N}\subseteq H(U,Y)$ and $\{h_n\}_{n\in\N}\subseteq H(U,Z)$ such that $f_n(\qu)=g_n(\qu)+h_n(\qu)$ for all $\qu\in U$ and $n\in\N$. Moreover by proposition \ref{1.2.4} we have $\sa(A/Y)\subseteq\fa(A|Y)\subseteq E$. In particular, we see that, for every $\qu\in\partial E$, the operator $R_\qu(A/Z)$ is invertible in the quotient space $\vr/Z$. By compactness and continuity, we obtain a constant $c>0$ such that
 $$\|R_\qu(A/Z)^{-1}\|\leq c\quad\text{for all}~~~\qu\in\partial E.$$
 Let $Q:\vr\longrightarrow\vr/Z$ be the natural quotient mapping. Then we obtain (as in equation \ref{FE1} of proposition \ref{1.2.4})
 $$Qg_n(\qu)=R_\qu(A/Z)^{-1}QR_\qu(A)f_n(\qu)\quad\text{for all}~~\qu\in\partial E.$$
 Therefore, as quotient map is continuous and $\|Q\|\leq 1$,
 \begin{equation}\label{DE4}
 \|Qg_n(\qu)\|\leq c\|R_\qu(A)f_n(\qu)\|\quad\text{for all}~~\qu\in\partial E.
 \end{equation}
 By the assumptions, equation \ref{DE3}, on the functions $f_n$, inequality  \ref{DE4} implies that the continuous right slice-regular functions $Q\circ g_n\in H(U,\vr/Z)$ converges to zero on $\partial E$, and therefore,  by the maximum modulus principle, uniformly on $E$. Since we know from proposition \ref{1.2.2} part (a) $H(U,\vr/Z)$ can be canonically identified with $H(E,\vr)/H(E,Z)$, we obtain functions $k_n\in H(E,Z)$ such that
 \begin{equation}\label{DE5}
 g_n+k_n\longrightarrow 0~~~~\text{as}~~~n\rightarrow\infty\quad\text{locally uniformly on}~E,
 \end{equation}
 and hence on $D$. Since
 \begin{equation}\label{DE6}
 f_n=g_n+h_n=(g_n+k_n)+(h_n-k_n)~~~\text{on}~~D~~\text{for all}~~n\in\N,
 \end{equation}
 it remains to be seen that $h_n-k_n\longrightarrow 0$ as $n\rightarrow\infty$ on $D$. Because $\sa(A|Z)\cap  D=\emptyset$, there exists a constant $d>0$ such that $\|R_\qu(A|Z)^{-1}\|<d$ for all $\qu\in D$. Since both $h_n$ and $k_n$ map into $Z$, and since $h_n=f_n-g_n$, we obtain
 \begin{eqnarray*}
 \|(h_n-k_n)(\qu)\|&\leq& d\|R_\qu(A)(h_n-k_n)(\qu)\|\quad\text{because}~~R_\qu(A|Z)^{-1}R_\qu(A)=\mathbb{I}_{Z}\\
 &\leq&d\|R_\qu(A)f_n(\qu)\|+d\|R_\qu(A)(g_n+k_n)(\qu)\|\quad\text{by equation \ref{DE6}}
 \end{eqnarray*}
 for all $\qu\in D$ and $n\in\N$. Thus, by equations \ref{DE3} and \ref{DE5}, $h_n-k_n\longrightarrow 0$ as $n\rightarrow\infty$ uniformly on $D$ as required.
\end{proof}
\begin{remark}\label{DR1}
	 The restriction of an operator with property ($\beta$) to a closed invariant subspace certainly has property ($\beta$). 
\end{remark}
%%%%%%%%%%%%%%%%%%%%%%%%%%%%%%

%%%%%%%%%%%%%%%%%%%%%%
Let $A\in\B(\vr)$ and $\phi\in\vr$. We are interested in a continuous right slice-regular function $f:U\longrightarrow\vr$ of the equation $R_\qu(A)f(\qu)=\phi$ on a suitable open subset $U\subseteq\quat$. On the resolvent set $\rho_S(A)$, there is a unique solution $f(\qu)=R_\qu(A)^{-1}\phi$ valid for all $\qu\in\rho_S(A)$. However, it is possible to obtain, for certain $\phi\in\vr$, continuous right slice-regular functions  of the equation $R_\qu(A)f(\qu)=\phi$ on an open set that contains points of the S-spectrum $\sa(A)$. The uniqueness of the continuous right slice-regular function is a non-trivial issue which is addressed in the next definition.\\
In \cite{Jo} (see page 311), the single valued extension property, local S-resolvent set and local S-spectrum are defined in terms of slice hyperholomorphic extension of the operator $R_\qu(A,\phi)=R_\qu(A)^{-1}(\phi\oqu-A\phi)$ on axially symmetric open sets containing $\rho_S(A)$ of $\quat$ . However, we stay with certain straightforward extensions of the complex definitions to quaternions.
%%%%%%%%%%%%%%%%%%%%%%%
\begin{definition}\label{1.2.9}
	An operator $A\in\B(\vr)$ has the single-valued extension properly, abbreviated SVEP, at $\q_0\in\quat$ if for every open neighborhood $U\subseteq\quat$ of $\qu_0$, the only continuous right slice-regular solution $f:U\longrightarrow\vr$ of the equation $R_\qu(A)f(\qu)=0$ for all $\qu\in U$ is the zero function on $U$. The operator $A$ is said to have the SVEP if $A$ has the SVEP at every point $\qu\in\quat$.
\end{definition}
%%%%%%%%%%%%%%%%%%%%%%%%%%%%%%%%%%%%%%%%%
\begin{remark}\label{DR2}Let $A\in\B(\vr)$.
	\begin{enumerate}
		\item[(a)] In terms of the operators considered in proposition \ref{1.2.6}, the condition   $R_\qu(A)f(\qu)=0$  for every open set $U\subseteq\quat$ means that the operator $A_U$ is injective in $H(U,\vr)$. Therefore, property ($\beta$) implies SVEP.
		\item[(b)] By part (a) and by theorem \ref{1.2.7}, all decomposable operators have SVEP.
		\item[(c)]We have, see proposition 3.1.9 in \cite{Jo}, $\ker(R_\qu(A))\not=\{0\}$ if and only if $\qu$ is a right eigenvalue of $A$. Thus, it is clear that if the set of eigenvalues of $A$ has empty interior, then $A$ has SVEP.
		\item[(d)](Theorem 4.17 in \cite{BT}) Let $A : D(A)\subseteq\vr\longrightarrow\vr$ be a densely defined right H-linear closed symmetric operator with the property that $i\cdot\phi,j\cdot\phi, k\cdot\phi\in D(A),$ for all $ \phi\in D(A).$ If the operators $i\cdot A, j\cdot A$ and $k\cdot A$ are anti-symmetric, then $A$ is self-adjoint if and only if the spherical spectrum $\sa(A)\subseteq\R$. The multiplications by $i,j,k$ are left multiplications in $\vr$.\\
		Under the set up of theorem 4.17 in \cite{BT}, we can show that (see the proof of theorem 4.17 in \cite{BT})
		$$\|R_\qu(A)\phi\|\geq (q_1^2+q_2^2+q_3^2)\|\phi\|,$$
		where $\qu=q_0+iq_1+jq_2+kq_3\in\quat$ and $\phi\in D(A^2)$. Thus, in this set up, if $A$ has real spectrum then $\kr(R_\qu(A))=\{0\}$, and hence $A$ has SVEP.
		\item[(e)] On the other hand, the set of right eigenvalues of an operator with SVEP may have non-empty interior (see example \ref{Ex3} below)
	\end{enumerate}
\end{remark}
%%%%%%%%%%%%%%%%%%%%%%%
According to the following proposition all non-invertible surjective operators will lack SVEP and hence lack property ($\beta$).
\begin{proposition}\label{1.2.10} If $A\in\B(\vr)$ is surjective and has SVEP, then $A$ is invertible.
\end{proposition}
\begin{proof}
	Since $A$ is surjective, by the open mapping theorem, there exists a constant $c>0$ with the property that, for every $\psi\in\vr$, there is a $\phi\in\vr$ such that $A\phi=\psi$ and $\|\phi\|\leq c\|\psi\|$. To prove $A$ is injective, consider an arbitrary $\phi_0\in\kr(A)$, and choose recursively $\phi_n\in\vr$ such that $A\phi_n=\phi_{n-1}$ and $\|\phi_n\|\leq c\|\phi_{n-1}\|$ for each $n\in\N$. Then we have $\|\phi_n\|\leq c^n\|\phi_0\|$ for all $n\in\N$ and hence $\displaystyle\lim_{n\rightarrow\infty}\sup\|\phi_n\|^{1/n}\leq c$. Hence, by theorem \ref{T5}, for each $\qu\in B_\quat(0,1/c)$, the series
	$$f(\qu)=\sum_{n=0}^\infty\phi_n\oqu^n$$
	converges uniformly on $B_\quat(0,1/c)$. For each fixed $\qu\in B_\quat(0,1/c)$, $f(\qu)\in\vr$. By repeating the above process we can obtain another sequence $\psi_0=f(\qu)$ and $A\psi_n=\psi_{n-1}$ with $\|\psi_n\|\leq d\|\psi_0\|$ for all $n\in\N$, where $d>0$ is some constant. Define
	$$g(\qu)=\sum_{n=1}^\infty\psi_n\qu^{n-1},$$
	which converges locally uniformly on $B_\quat(0,r)$, where $r=\min\{1/c,1/d\}$, and hence defines a continuous right-regular function on $B_\quat(0,r)$. We have
	$$(A-\qu\Iop)g(\qu)=\sum_{n=1}^\infty\psi_{n-1}\qu^{n-1}-\sum_{n=1}^\infty\psi_n\qu^n=\psi_0=f(\qu).$$
	Hence
$$R_\qu(A)g(\qu)=(A-\oqu\Iop)f(\qu)=A\phi_0+\sum_{n=1}^\infty\phi_{n-1}\oqu^n-\sum_{n=0}^\infty\phi_n\oqu^{n+1}=A\phi_0=0.$$
Therefore, by SVEP $g=0$ on	 $B_\quat(0,r)$. In particular, for $\qu=0$, we have $0=g(0)=\psi_0=f(0)=\phi_0$. This proves the injectivity of $A$, and hence $A$ is invertible.
\end{proof}	
%%%%%%%%%%%%%%
The following example provides an operator without SVEP, and hence without property ($\beta$).
\begin{example}\label{SE1}
		 The  space
	$$l^2(\N)=\left\{x:\N\longrightarrow\quat;~~i\mapsto x_i~~\vert~~\sum_{i=1}^\infty |x_i|^2<\infty\right\}$$
	with the inner product
	$$\langle x|y\rangle=\sum_{i=1}^\infty \overline{x_i}~y_i\quad \text{for all}~~x=(x_1,x_2,\cdots), y=(y_1,y_2,\cdots)\in l^2(\N)$$
	is a right quaternionic Hilbert space.
	Consider the unilateral left shift $S$ on $l^2(\N)$ given by
	$$S(x_1,x_2,x_3,\cdots)=(x_2,x_3,\cdots)\quad\text{for all}~~~(x_1,x_2,x_3,\cdots)\in l^2(\N).$$
	Since for any given $(x_1,x_2,x_3,\cdots)\in l^2(\N)$, we have $(0,x_1,x_2,x_3,\cdots)\in l^2(\N)$ such that
		$$S(0,x_1,x_2,x_3,\cdots)=(x_1,x_2,x_3,\cdots).$$
		Therefore $S$ is surjective. Also $\kr(S)=\{(x_1,x_2,x_3,\cdots)\in l^2(\N)~~|~~x_i=0~~\forall i\geq 2\}\not=\{0\}$. Thus $S$ is not invertible. Hence, according to proposition \ref{1.2.10}, the operator $S$ does not have SVEP and hence lack property ($\beta$).
	\end{example}
%%%%%%%%%%%%%%%%%%%%%%%%%%%
\subsection{Local S-spectrum}
\begin{definition}\cite{Ka}\label{L}
	Let $A\in\B(\vr)$ the local S-resolvent set $\rho_A^S(\phi)$ of $A$ at a point $\phi\in\vr$ is defined as the union of all open subsets $U$ of $\quat$ for which there is a continuous  right slice-regular function $f:U\longrightarrow\vr$ which satisfies
	$$R_\qu(A)f(\qu)=\phi,\quad\text{for all}~~\qu\in U.$$
\end{definition}
If the function $f$ is defined on the set $\lr(\phi)$ then it is called a local resolvent function of $A$ at $\phi$.
	The local S-spectrum $\ls(\phi)$ of $A$ at $\phi$ is then defined as
	$$\ls(\phi)=\quat\setminus\lr(\phi).$$

%%%%%%%%%%%%%%%%%%%%%
\begin{remark}\label{DR3} Let $A\in\B(\vr)$ and $\phi\in\vr$. 
	\begin{enumerate}
		\item [(a)] Clearly $\lr(\phi)$ is an open subset of $\quat$ given by the union of the domains  of all the local resolvent functions.
		\item[(b)] $\rho_S(A)\subseteq\lr(\phi)$ and $\ls(\phi)\subseteq\sa(A)$. Since $\ls(\phi)$ is closed, it is compact \cite{Ka}.
		\item[(c)] Since $\rho_S(A)\subseteq\lr(\phi)$, the continuous right slice-regular solutions occurring in the definition of local resolvent set may be thought of as local extensions of the function $R_\qu(A)^{-1}\phi$. But there is no uniqueness implied. It is evident that, as in the complex case, the local continuous right slice-regular solutions will be unique if and only if $A$ has SVEP. In this case, they define a continuous right slice-regular function on all of $\lr(\phi)$, which is the maximal continuous right slice-regular extension of $R_\qu(A)^{-1}\phi$ from $\rho_S(A)$ to $\lr(\phi)$.
		\item[(d)] As in the complex case, we call, for an operator $A\in\B(\vr)$ with SVEP and an arbitrary $\phi\in\vr$, the unique continuous right slice-regular solution $f:\lr(\phi)\longrightarrow\vr$ of the equation $R_\qu(A)f(\qu)=\phi$ for all $\qu\in\lr(\phi)$  the {\em local resolvent function} for $A$ at $\phi$.
		
	\end{enumerate}
\end{remark}
%%%%%%%%%%%%%%%%%%%
The following  example illustrates that, for a large class of multiplication
operators, the local spectrum is closely related to the notion of support.
\begin{example}\label{Ex2}
	Let $\Omega$ be a compact Hausdorff space, and let $$C_{\quat}(\Omega)=\{f:\Omega\to\quat~:~f\text{~is continuous}\}.$$ Then $C_\quat(\Omega)$ is a right linear Banach space which is endowed with pointwise operations and the supremum norm. Let $A$ be the operator of multiplication on $C_\quat(\Omega)$ by an arbitrary function $g\in C_\quat(\Omega)$. That is, $Af(x):=f(x)g(x),~\forall\,x\in\Omega$. Now for any $\qu\in\quat$, 
	\begin{eqnarray*}
		R_\qu(A)\text{~~is onto~~}&\Leftrightarrow&\text{~~for given~~}h\in C_\quat(\Omega),\,\exists\,f\in C_\quat(\Omega)\text{~~such that~~}R_\qu(A)f=h\\
		&\Leftrightarrow&f(x)=h(x)(g(x)^2-2\text{Re}(\qu)g(x)+|\qu|^2)^{-1},\,\forall\,x\in\Omega\\
		&\Leftrightarrow& \qu,\overline{\qu}\notin g(\Omega)\\
		&\Leftrightarrow&\qu\notin g(\Omega)\cup g(\Omega)^\ast.
	\end{eqnarray*}
	That is, $R_\qu(A)$ is not onto if and only if $\qu\in g(\Omega)\cap g(\Omega)^\ast$. Also we have for any $\qu\in\quat$,
	\begin{eqnarray*}
		R_\qu(A)\text{~~is one to one~~}&\Leftrightarrow& \ker{R_\qu(A)}=\{0\}\\
		&\Leftrightarrow&R_\qu(A)f\neq0,\text{~~for all~~}0\neq f\in C_\quat(\Omega)\\ 
		&\Leftrightarrow& R_\qu(A)f=f(g-\qu)(g-\overline{\qu})\neq0,\,\forall\,0\neq f\in C_\quat(\Omega)\\.
		&\Leftrightarrow& g(x)\neq\qu\text{~~or~~}g(x)\neq\overline{\qu},\,\forall\,x\in\Omega\\
		&\Leftrightarrow& \qu,\overline{\qu}\notin g(\Omega)\\
		&\Leftrightarrow&\qu\notin g(\Omega)\cup g(\Omega)^\ast.
	\end{eqnarray*}
	That is, $$\sigma_S(A)=g(\Omega)\cup g(\Omega)^\ast.$$ Also note that for this multiplication operator $A$, we have $$\sigma_{pS}(A)=\sigma_S(A)=g(\Omega)\cup g(\Omega)^\ast.$$
	We claim that
	\begin{enumerate}
		\item $A$ is decomposable,
		\item $\ls(f)=g(\text{supp}\, f)$, for all $f\in C_\quat(\Omega)$;
	\end{enumerate}
	where $\text{supp}\, f$ denotes the \textit{support} of the function $f\in C_\quat(\Omega)$, and is defined as follows $$\text{supp}\, f:=\overline{\{x\in\Omega~:~f(x)\neq0\}}.$$
	To verify the claim (1), let $\{U_1,U_2\}$ be arbitrary open cover of $\quat$, that is, $\quat=U_1\cup U_2$. Then $\{g^{-1}(U_1),g^{-1}(U_2)\}$ is an open cover of $\Omega$, as $g$ is a continuous function. Since the compact Housdorff spaces are normal, and by the normality of $\Omega$, $C_\quat(\Omega)$ admits a \textit{partition of unity}, in the sense that for the open cover $\{g^{-1}(U_1),g^{-1}(U_2)\}$ of $\Omega$, there are functions $e_1,e_2\in C_\quat(\Omega)$ for which 
	\begin{center}
		$e_1+e_2\equiv1$ on $\Omega$ and supp $e_k\subseteq g^{-1}(U_k)$ for $k=1,2$.
	\end{center}
	For $k=1,2$, let $X_k=\{f\in C_\quat(\Omega)~:~\text{supp}\, f\subseteq\text{supp} \,e_k\}$. Now we shall show the following:
	\begin{itemize}
		\item [(a)] $X_k$ is  $A$-invariant, for $k=1,2$,
		\item [(b)] $X_k$ is a closed linear subspace of $C_\quat(\Omega)$, for  $k=1,2$.
		%\item [(c)] $\sigma_S(A|X_k)\subseteq U_k$ for all $k=1,2$.
	\end{itemize}
	Let $k=1,2$. Take $h\in A(X_k)$, then there exists $f\in X_k$ such that \begin{center}
		$h(x)=A(f)(x)=f(x)g(x)$, for all $x\in\Omega$.
	\end{center} If $x\in\text{supp}\,A(f)$, then there exists a sequence $\{x_n\}$ in $\Omega$ such that $$\{x_n\}\subseteq\{x\in\Omega~:~f(x)g(x)\neq0\}\subseteq\{x\in\Omega~:~f(x)\neq0\}$$ and $x_n\longrightarrow x$ as $n\longrightarrow\infty$. Thus $x\in\text{supp}\,f\subseteq\text{supp}\,e_k$ as $f\in X_k$. That is, $$\text{supp}\,h=\text{supp}\,A(f)\subseteq\text{supp}\,e_k.$$ Hence assertion (a) follows. To prove the statement (b), let $k=1,2$ and choose $f\in\overline{X_k}$. Then there is a sequence $\{f_n\}$ in $X_k$ such that $f_n\longrightarrow f$ as $n\longrightarrow\infty$. Let $x\in\text{supp}\,f$, then there is a sequence $\{x_m\}$ in $\Omega$ such that $f(x_m)\neq0$, for all $m\in\mathbb{N}$ and $x_m\longrightarrow x$ as $m\longrightarrow\infty$. Now for each $m\in\mathbb{N}$, $f_n(x_m)\longrightarrow f(x_m)\neq0$ as $n\longrightarrow\infty$. This implies $\{x_m\}\subseteq\{x\in\Omega~:~f_n(x)\neq0\}$. Thus $x\in\text{supp}\,f_n\subseteq\text{supp}\,e_k$, for all $n\in\mathbb{N}$ as  $\{f_n\}\subseteq X_k$. That is, $\text{supp}\,f\subseteq\text{supp}\,e_k$. Therefore $f\in X_k$, and hence the statement holds true.\\
	Furthermore, for each $k=1,2$, $\sigma_S(A|X_k)\subseteq U_k$. Indeed, take arbitrarily $\qu\in\quat\smallsetminus U_k$ with $\overline{\qu}\in\quat\smallsetminus U_k$. Assume that $R_\qu(A)f=0$ for some $0\neq f\in X_k$. Then for each $x\in\Omega$, $f(x)(g(x)-\qu)(g(x)-\overline{\qu})=0$. Now if $f(x)\neq0$, then $x\in\text{supp}\,f\subseteq\text{supp}\,e_k\subseteq g^{-1}(U_k)$ and $g(x)=\qu$ or $\overline{\qu}$. This implies $g(x)\in U_k$ and $g(x)=\qu$ or $\overline{\qu}$, which is a contradiction. Hence, $R_\qu(A)$ is one to one. To see $R_\qu(A)$ is onto, for an arbitrary $f\in X_k$, define
	$$h(x):=\begin{cases}
	f(x)(g(x)^2-2\text{Re}(\qu)g(x)+|\qu|^2)^{-1}, & \text{~~if~~}x\in g^{-1}(U_k)\\
	0, & \text{~~if~~}x\in\Omega\smallsetminus\text{supp}\,e_k.
	\end{cases}$$
	Then $h\in C_\quat(\Omega)$ is the only solution of the equation $R_\qu(A)h=f$. This proves that $R_\qu(A)$ is invertible and $\qu\in\rho_S(A|X_k)$. Hence the inclusion $\sigma_S(A|X_k)\subseteq U_k$ follows for $k=1,2$. Moreover, every $f\in C_\quat(\Omega)$, admits the decomposition \begin{center}
		$f=fe_1+fe_2$ with $e_kf\in X_k$ for $k=1,2$.
	\end{center}
	Therefore $C_\quat(\Omega)=X_1+X_2$, which verifies the claim (1). To prove the claim (2), let $f\in C_\quat(\Omega)$. take $x\in \text{supp}\,f$ and if $g(x)\in\lr(f)$, the there exists an open set $U\subseteq\quat$ and a continuous right slice-regular function $F:U\longrightarrow C_\quat(\Omega)$ such that 
	$$R_\qu(A)F(\qu)=f,~\forall\,\qu\in U.$$ That is, for each $y\in\Omega$, 
	$$F(\qu)(y)(g(y)-\qu)(g(y)-\overline{\qu})=f(y).$$ This equation implies that $f(y)=0$ as $g(y)=\qu$ or $\overline{\qu}$. That is,
	$$g^{-1}(\lr(f))\subseteq\{y\in\Omega~:~f(y)=0\}.$$
	Thus $$g(\{x\in\Omega~:~f(x)\neq0\})\subseteq\ls(f).$$
	Let $x\in g(\text{supp}\,f)$, then $x=g(z)$ for some $z\in\text{supp}\,f$. Since $z\in\text{supp}\,f$, we have  a sequence $\{z_n\}\subseteq\{x\in\Omega~:~f(x)\neq0\}$ such that $z_n\longrightarrow z$ as $n\longrightarrow\infty$. The continuity of $g$ admits that
	$$g(z_n)\longrightarrow g(z)=x\text{~~as~~}n\longrightarrow\infty.$$
	But $\{g(z_n)\}\subseteq\ls(f)$. Thus $x\in\overline{\ls(f)}=\ls(f)$ as $\ls(f)$ is compact (closed). Therefore, $$g(\text{supp}\,f)\subseteq\ls(f).$$
	To see the opposite inclusion, Let $\qu\notin g(\text{supp}\,f)$. Choose $r>0$ such that $$\nabla_\quat(\qu,r)\cap g(\text{supp}\,f)=\emptyset.$$ Then for any $\pu\in B_\quat(\qu,r)$ and $x\in\Omega$, let
	$$h_\pu(x):=\begin{cases}
	f(x)(g(x)^2-2\text{Re}(\pu)g(x)+|\pu|^2)^{-1}, & \text{~~if~~} g(x)\notin\nabla_\quat(\qu,r)\\
	0, & \text{~~if~~}g(x)\notin\text{supp}\,f.
	\end{cases}$$ 
	Now for $\pu\in B(\qu,r)$ we have obtained a well-defined  function $h_\pu\in C_\quat(\Omega)$ with the property that $R_\pu(A)h_\pu=f$. It is easily seen that the mapping $\pu\longrightarrow h_\pu$ is continuous and right regular on $B(\qu,r)$. Thus $\qu\in\lr(f)$. Hence the claim (2) is verified.
\end{example}

%%%%%%%%%%%%
\begin{example}\label{Ex3}
In example \ref{Ex2} let $\Omega$ be a subset of $\quat$ and it has non-empty interior. Set $g(\qu)=\qu$ for all $\qu\in\Omega$. Since the operator $A$ considered in example \ref{Ex2} is decomposable, by remark \ref{DR2} part (b), $A$ has SVEP, while $\sigma_{pS}(A)=\Omega\cup\Omega^*$ has nonempty interior. That is, the set of right eigenvalues of $A$ with SVEP has non-empty interior.
\end{example}
%%%%%%%%%%%%%%%%%%%%%%%%%%%%%%
In the complex local spectral theory, the boundedness of the local resolvent function is also of interest. In the similar manner, let us discuss the boundedness of the local resolvent function in the quaternionic setting.\\
Let $\qu\in\partial\sa(A)$ and $A\in\B(\vr)$, then there exists a sequence $\{\qu_n\}\subseteq\rho_S(A)$ such that $\qu_n\longrightarrow\qu$, and hence $R_{\qu_n}(A)^{-1}\longrightarrow R_\qu(A)^{-1}$ as $n\rightarrow\infty$	and $\displaystyle\sup_n\|R_{\qu_n}(A)\|=\infty$. Thus $\|R_\qu(A)^{-1}\|\longrightarrow\infty$ as $n\rightarrow\infty$ (see the proof of theorem 5.4 in \cite{Fr} for details). Therefore, by the uniform boundedness principle (which holds for quaternion \cite{Jo}), for some $\phi\in\vr$, the function $R_\qu(A)^{-1}\phi$ is unbounded in $\rho_S(A)$.\\
Suppose that $\sa(A)$ is countable (compact normal operators in $\vr$ have such property, see example 14.3.10 in\cite{Jo}). We also have For $\phi\in\vr$, $\ls(\phi)\subseteq\sa(A)$ is compact. Also every countable compact subset of $\quat$ has at least one isolated point. Every isolated point of the local spectrum is a non-removable singularity of the continuous local resolvent function $h_\qu:\lr(\phi)\longrightarrow\vr$. Hence, in such a case, all non-trivial local resolvent functions are unbounded. However, for a large class of multiplication operators, the following example provides necessary and sufficient condition for the local resolvent function to be bounded.
%%%%%%%%%%%%%%%%%%%
\begin{example}\label{Ex4}
	Consider the  example given in example \ref{Ex2}. That is $\Omega$ be a non-empty compact Hausdorff space and $A$ be the  operator of multiplication on $C_\quat(\Omega)$ by a given function $g\in C_\quat(\Omega)$. From example \ref{Ex2} we know that $\sa(A)=g(\Omega)\cup g(\Omega)^*$. Also $A$ is decomposable, hence, by remark \ref{DR2}, $A$ has SVEP. Therefore, for every $f\in C_\quat(\Omega)$, the local resolvent function for $A$ at $f$ is uniquely determined and defined on the entire local resolvent set $\lr(f)$.\\
	{\em Claim:} There exists a non-trivial $f\in C_\quat(\Omega)$ for which the local resolvent function of $A$ at $f$ is bounded on $\lr(f)$ if and only if $g(\Omega)$ has nonempty interior.\\
	Suppose that $g(\Omega)$ has non-empty interior. Consider the function $f\in C_\quat(\Omega)$ given by
	$$f(x)=\text{dist}(g(x), \quat\setminus(\text{int}(g(\Omega)\cup g(\Omega)^*)))^2$$
	For arbitrary $x\in\Omega$, when $g(x)\in g(\Omega)\cup g(\Omega)^*$, $f(x)>0$, that is, $f$ is non-trivial. Moreover, from example \ref{Ex2}, we have $\ls(f)=g(\text{supp} f)$,  and hence $\lr(f)\subseteq\quat\setminus(\text{int}(g(\Omega)\cup g(\Omega)^*))$. Let $\qu\mapsto h_\qu$ denote the local resolvent function for $A$ at $f$, then
	$$|h_\qu(x)|=\frac{|f(x)|}{|g(x)-\qu|~|g(x)-\oqu|}\quad\text{for all }~~\qu\in\lr(f)~~\text{and}~~x\in\Omega.$$
	Let $x\in\Omega$ be given, and consider any $\qu\in\lr(f)$ for which $\qu\not=g(x)$ and $\oqu\not=g(x)$. Since $\qu\in\quat\setminus(\text{int}g(\Omega)\cup\text{int} g(\Omega)^*)$, by the definition of $f$, we have $|h_\qu(x)|\leq 1$. By the continuity of the local resolvent function $|h_\qu(x)|\leq 1$ for all $\qu\in\lr(f)$. Thus the function $\qu\mapsto h_\qu$ is bounded on $\lr(f)$.\\
	For the converse, since $\ls(f)=g(\text{supp}f)$, it suffices to show that, for any non-zero $f\in C_\quat(\Omega)$ for which $\ls(f)$ has empty interior, the corresponding local resolvent function $\qu\mapsto h_\qu$ is unbounded on $\lr(f)$. Let $x\in\Omega$ be a point for which $f(x)\not=0$. Since $\text{int}\ls(f)=\emptyset$, there exists a sequence $\{\qu_n\}\subseteq\lr(f)$ such that $\qu_n\longrightarrow g(x)$ as $n\rightarrow\infty$. Hence
	$$|h_{\qu_n}(x)|=\frac{|f(x)|}{|g(x)-\qu_n|~|g(x)-\oqu_{n}|}\longrightarrow\infty$$
	as $n\rightarrow\infty$, thus the local resolvent function is unbounded.
\end{example}

%%%%%%%%%%%%%%%%%%%%%%%%%%%%%%
\begin{definition}\cite{Ka}\label{L1} Let $A\in\B(\vr)$ and $F\subseteq\quat$. The local  S-spectral subspace of $A$ associated with $F$ is defined by
	$$V_A(F)=\{\phi\in\vr~~|~~\ls(\phi)\subseteq F\}.$$
\end{definition}
\begin{definition}\label{L2}\cite{Ka} Let $A\in\B(\vr)$ and $F\subseteq\quat$ be a closed subset. The set $\va(F)$ consists of all $\phi\in\vr$ for which there exists a right slice-regular function $f:\quat\setminus F\longrightarrow\vr$ that satisfies $R_\qu(A)f(\qu)=\phi$ for all $\qu\in\quat\setminus F$. The set $\va(F)$ is called the global S-spectral subset of $A$ associated with the set $F$.
\end{definition}
\begin{remark}\label{DR4}
	\begin{enumerate}
		\item [(a)] In the complex theory the counterparts of $V_A(F)$ and $\va(F)$ play significant role in the theory of spectral decompositions.
		\item[(b)] If $F\subseteq G\subseteq\quat$ then clearly $V_A(F)\subseteq V_A(G)$.
		\item[(b)] Immediately  from the definition, for every collection of subsets $\{F_\alpha\subseteq\quat~~|~~\alpha\in I\}$, $I$ is an index set, $$V_A\left(\bigcap_{\alpha\in I}F_\alpha\right)=\bigcap_{\alpha\in I}V_A(F_\alpha).$$ 
	\end{enumerate}
\end{remark}
%%%%%%%%%%%%%%%%%%%%%%%%%%%%%%%%%%%%%
\begin{proposition}\label{Ai-P57}
	Let $A\in\B(\vr)$, $\phi,\psi\in\vr$ and $\pu,\qu\in\quat$ then we have
	\begin{enumerate}
		\item [(a)]$\ls(0)=\emptyset$;
		\item[(b)] $\ls(\phi\qu+\psi\pu)\subseteq\ls(\phi)\cup\ls(\psi)$;
		\item[(c)]$\ls(B\phi)\subseteq\ls(\phi)$ for every $B\in\B(\vr)$ which commutes with $A$.
	\end{enumerate}
\end{proposition}
\begin{proof}
	See proposition 6.9 in \cite{Ka}.
\end{proof}
%%%%%%%%%%%%%%%%%%%%
\begin{proposition}\label{2.1.14}
	Let $A\in\B(\vr)$, $\phi\in\vr$ and $U$ be an open subset of $\quat$. Suppose that $f:U\longrightarrow\vr$ is a continuous right slice-regular function for which $R_\pu(A)f(\pu)=\phi$ for all $\pu\in U$. Then $\ls(\phi)\subseteq\ls(f(\qu))$ for all $\qu\in U$.
\end{proposition}
\begin{proof}
Let $\mathbf{s}\not\in\ls(f(\qu))$, then $\mathbf{s}\in\lr(f(\qu))$. Thus, there exists an open neighborhood $U_\mathbf{s}$ of $\mathbf{s}$ such that $h:U_\mathbf{s}\longrightarrow\vr$, a continuous right slice-regular function, satisfying $R_\pu(A)h(\pu)=f(\qu)$	for all $\pu\in U_\mathbf{s}$. Then,
$$R_\pu(A)R_\qu(A)h(\pu)=R_\qu(A)R_\pu(A)h(\pu)=R_\qu(A)f(\qu)=\phi\quad\text{for all}~~\pu\in U_\mathbf{s}.$$
Therefore $\mathbf{s}\in\lr(\phi)$, and hence $\mathbf{s}\not\in\ls(\phi).$
\end{proof}
%%%%%%%%%%%%%%%%%%%%%%%%%%%%%%%%
\begin{theorem}\label{2.4-Ai-II}
	Let $A\in\B(\vr,\ur)$ and $B\in\B(\ur,\vr)$. Then we have the following.
	\begin{enumerate}
		\item [(a)]For every $\phi\in\vr$ the following inclusions hold:
		$$\sigma_{AB}^S(A\phi)\subseteq\sigma_{BA}^S(\phi)\subseteq\sigma_{AB}^S(A\phi)\cup\{0\}.$$
		\item[(b)]If $A$ is injective, then $\sigma_{BA}^S(\phi)=\sigma_{AB}^S(A\phi)$ for all $\phi\in\vr$.
		\item [(c)]For every $\psi\in\ur$ the following inclusions hold:
		$$\sigma_{BA}^S(B\psi)\subseteq\sigma_{AB}^S(\psi)\subseteq\sigma_{BA}^S(B\psi)\cup\{0\}.$$
		\item[(d)]If $B$ is injective, then $\sigma_{BA}^S(B\psi)=\sigma_{AB}^S(\psi)$ for all $\psi\in\ur$.
	\end{enumerate}
\end{theorem}
\begin{proof}
	(a)~~Let $\qu\not\in\sigma_{BA}^S(\phi)$, then there is an open neighborhood $U_\qu\subseteq\quat$ of $\qu$ and a continuous right slice-regular function $f:U_\qu\longrightarrow\vr$ such that $R_\pu(BA)f(\pu)=\phi$ for all $\pu\in U_\qu$. Hence, $AR_\pu(BA)f(\pu)=A\phi$ for all $\qu\in U_\qu$. That is,
	$$(ABABA-2\text{Re}(\pu)ABA+|\pu|^2A)f(\pu)=A\phi\quad\text{for all}~~\pu\in U_\qu,$$
	and hence $R_\pu(AB)A\circ f(\pu)=A\phi$ for all $\pu\in U_\qu$. Since $A\circ f(\pu)$ is a continuous right slice-regular function on $U_\qu$, $\qu\in\rho_{AB}^S(A\phi)$, and therefore $\qu\not\in\sigma_{AB}^S(A\phi)$. Hence,
	\begin{equation}\label{DE7} 
	\sigma_{AB}^S\subseteq\sigma_{BA}^S(\phi).
	\end{equation}
	To show the second inclusion, let $\qu\not\in\sigma_{AB}^S(A\phi)\cup\{0\}$, then there exists an open neighborhood $V_\qu\subseteq\quat$ of $\qu$ and a continuous right slice-regular function $g:V_\qu\longrightarrow\ur$ such that \begin{equation}\label{DE8}
	R_\pu(AB)g(\pu)=A\phi\quad\text{ for all}~~ \pu\in V_\qu.
	\end{equation}
	set
	$$h(\pu)=\frac{1}{|\pu|^2}(\phi-BABg(\pu)+2\text{Re}(\pu)Bg(\pu))\quad\text{for all}~~\pu\in V_\qu,$$
	which is a continuous right slice-regular function on $V_\qu$. Now
	\begin{eqnarray*}
		&	&R_\pu(BA)h(\pu)=\\&=&((BA)^2-2\text{Re}(\pu)BA+|\pu|^2\Iop)(\frac{1}{|\pu|^2}(\phi-BABg(\pu)+2\text{Re}(\pu)Bg(\pu)))\\
		&=&\frac{1}{|\pu|^2}(BA)^2\phi-\frac{1}{|\pu|^2}(BA)^3Bg(\pu)+\frac{2\text{Re}(\pu)}{|\pu|^2}(BA)^2Bg(\pu)-\frac{2\text{Re}(\pu)}{|\pu|^2}BA\phi\\
		& &+\frac{2\text{Re}(\pu)}{|\pu|^2}(BA)^2Bg(\pu)
		-\frac{4(\text{Re}(\pu))^2}{|\pu|^2}BABg(\pu)+\phi-BABg(\pu)+2\text{Re}(\pu)Bg(\pu)\\
		&=&\frac{1}{|\pu|^2}BAB\left[(AB)^2g(\pu)-2\text{Re}(\pu)ABg(\pu)+|\pu|^2g(\pu)\right]-\frac{1}{|\pu|^2}(BA)^3Bg(\pu)\\
		& &+\frac{2\text{Re}(\pu)}{|\pu|^2}(BA)^2Bg(\pu)
		-\frac{2\text{Re}(\pu)}{|\pu|^2}B\left[(AB)^2g(\pu)-2\text{Re}(\pu)ABg(\pu)+|\pu|^2g(\pu)\right]\\
		& &+\frac{2\text{Re}(\pu)}{|\pu|^2}(BA)^2Bg(\pu)
		-\frac{4(\text{Re}(\pu))^2}{|\pu|^2}BABg(\pu)+\phi-BABg(\pu)+2\text{Re}(\pu)Bg(\pu)\\
		& &\quad\quad\quad\quad\quad\quad\quad\quad\quad\quad\quad\quad\quad\quad\quad\quad\text{where we used equation \ref{DE8}}\\
		&=&\phi.
			\end{eqnarray*}
		That is $R_\pu(BA)h(\pu)=\phi$ for all $\pu\in V_\qu$. Hence $\qu\in\rho_{BA}^S(\phi)$. That is $\qu\not\in\sigma_{BA}^S(\phi)$. Therefore $\sigma_{BA}^S(\phi)\subseteq\sigma_{AB}^S(A\phi)\cup\{0\}$, which completes the proof of (a).\\
		
		(b)~~Assume $\qu\not\in\sigma_{AB}^S(A\phi)$. There is no harm in assuming $\qu=0$. Thus, assume $0\not\in\sigma_{AB}^S(A\phi)$. Then there is a continuous right slice-regular function $g:U_0\longrightarrow\ur$, where $U_0\subseteq\quat$ is an open neighborhood of zero, such that $R_\pu(AB)g(\pu)=A\phi$ for all $\pu\in U_0$. For $\pu=0$ we have $(AB)^2g(0)=A\phi$ and from the injectivity of $A$ we get $BABg(0)=\phi$. Moreover,
		$$|\pu|^2g(\pu)=A\phi+2\text{Re}(\pu)ABg(\pu)-(AB)^2g(\pu)
		=A[\phi+2\text{Re}(\pu)Bg(\pu)-BABg(\pu)].$$
		That is,~ $\displaystyle g(\pu)=A[\frac{1}{|\pu|^2}(\phi+2\text{Re}(\pu)Bg(\pu)-BABg(\pu))].$ 
		Let $Ak(\pu)=\left[2\text{Re}(p)AB-(AB)^2\right]g(\pu)$. Then, as $\pu\rightarrow 0$ if and only if $|\pu|\rightarrow 0$,
		\begin{eqnarray*}
			-\partial_SAk(0)&=&-Ak'(0)
			 =-\lim_{\pu\rightarrow 0}(Ak(p)-Ak(0))|\pu|^{-2}\\
			&=&-\lim_{\pu\rightarrow 0}[2\text{Re}(\pu)ABg(\pu)-(AB)^2g(\pu)-(AB)^2g(0)]\pu|^{-2}\\
			&=&-\lim_{\pu\rightarrow 0}[2\text{Re}(\pu)ABg(\pu)-(AB)^2g(\pu)-A\phi]\pu|^{-2}\\
			&=&-\lim_{\pu\rightarrow 0}[2\text{Re}(\pu)ABg(\pu)-(AB)^2g(\pu)-(AB)^2g(\pu)+2\text{Re}(\pu)ABg(\pu)-|\pu|^2g(\pu)]\pu|^{-2}\\
			&=&\lim_{\pu\rightarrow 0}g(\pu)=g(0).
		\end{eqnarray*}
	On $U_0$ define,
	$$\displaystyle h(\pu)=\left\{\begin{array}{ccc}
			
	\frac{1}{|\pu|^2}(\phi+2\text{Re}(\pu)Bg(\pu)-BABg(\pu))&\text{if}&\pu\not=0\\-k'(0)&\text{if}&\pu=0.
	\end{array}\right.$$
	Then $h(\pu)$ is a continuous right slice-regular function on $U_0$ and for $\pu\not=0$ we have
	\begin{eqnarray*}
		& &A\left[R_\pu(BA)h(\pu)-\phi\right]=\\&=&A\left[((BA)^2-2\text{Re}(\pu)BA+|\pu|^2\Iop)(	\frac{1}{|\pu|^2}(\phi+2\text{Re}(\pu)Bg(\pu)-BABg(\pu)))-\phi\right]\\
		&=&\frac{1}{|\pu|^2}(AB)^2A\phi+\frac{4\text{Re}(\pu)}{|\pu|^2}(AB)^3g(\pu)-\frac{1}{|\pu|^2}(AB)^4g(\pu)-\frac{2\text{Re}(\pu)}{|\pu|^2}ABA\phi\\
		& &-\frac{4(\text{Re}(\pu))^2}{|\pu|^2}(AB)^2g(\pu)+2\text{Re}(\pu)ABg(\pu)-(AB)^2g(\pu)\\
		&=&\frac{1}{|\pu|^2}(AB)^2R_\pu(AB)g(\pu)+\frac{4\text{Re}(\pu)}{|\pu|^2}(AB)^3g(\pu)-\frac{1}{|\pu|^2}(AB)^4g(\pu)\\
		& &-\frac{2\text{Re}(\pu)}{|\pu|^2}ABR_\pu(AB)g(\pu)-\frac{4(\text{Re}(\pu))^2}{|\pu|^2}(AB)^2g(\pu)+2\text{Re}(\pu)ABg(\pu)-(AB)^2g(\pu)\\
		&=&0.
	\end{eqnarray*}
Also for $\pu=0$,
\begin{eqnarray*}
A[R_0(BA)h(\pu)-\phi]&=&A[(BA)^2k'(0)-BABg(0)]\\
&=&A[BAB(Ak'(0))-BABg(0)]=A[BABg(0)-BABg(0)]=0.
\end{eqnarray*}
That is, $A\left[R_\pu(BA)h(\pu)-\phi\right]=0$ for all $\pu\in U_0$. Since $A$ is injective, we get $R_\pu(BA)h(\pu)=\phi$ for all $\pu\in U_0$, and hence $0\not\in\sigma_{BA}^S(\phi)$. Therefore, together with part (a), we have $\sigma_{BA}^S(\phi)=\sigma_{AB}^S(A\phi)$ for all $\phi\in\vr$.\\

Proofs of (c) and (d) are similar to parts (a) and (b).	
\end{proof}
%%%%%%%%%%%%%%%%%%%%%%%
For an injective operator $T\in\B(\vr)$, the local spectra of $T\phi$ and $\phi$ coincide for any $\phi\in\vr$.
\begin{corollary}\label{2.5-Ai-II}
	Let $T\in\B(\vr)$ and $\phi\in\vr$. Then we have
	\begin{enumerate}
		\item [(a)]$\sigma_T^S(T\phi)\subseteq\sigma_T^S(\phi)\subseteq\sigma_T^S(T\phi)\cup\{0\}.$
		\item[(b)] If $T$ is injective, then $\sigma_T^S(T\phi)=\sigma_T^S(\phi)$.
	\end{enumerate}
\end{corollary}
\begin{proof}
	Take $A=T$ and $B=\Iop$ in proposition \ref{2.4-Ai-II}.
\end{proof}
%%%%%%%%%%%%%%%%%%%%%
Now we consider the case when $A,B\in\B(\vr)$ satisfy the operator equation $ABA=A^2$. Examples of such operators are given by $A=PQ$ where $P,Q\in\B(\vr)$ are idempotents. 
\begin{proposition}\label{2.7-Ai-II}
	Suppose that $A,B\in\B(\vr)$ satisfy $ABA=A^2$. Then we have, for all $\phi\in\vr$,
	\begin{enumerate}
		\item[(a)] $\ls(A\phi)\subseteq\sigma_{BA}^S(\phi)$;
		\item[(b)] $\sigma_{BA}^S(BA\phi)\subseteq\ls(\phi)$.
		\end{enumerate}
	\end{proposition}
\begin{proof}
(a)~~	Suppose that $\qu_0\not\in\sigma_{BA}^S(\phi)$, that is  $\qu_0\in\rho_{BA}^S(\phi)$. Then there exists an open neighborhood $U_0\subseteq\quat$, and a continuous right slice-regular function $f:U_0\longrightarrow\vr$, such that $R_\qu(BA)f(\qu)=\phi$ for all $\qu\in U_0$. Hence,
	\begin{eqnarray*}
		A\phi&=&A((BA)^2-2\text{Re}(\qu)BA+|\qu|^2\Iop)f(\qu)\\
		&=&(ABA^2-2\text{Re}(\qu)A^2+|\qu|^2A)f(\qu)\\
		&=&(A^2-2\text{Re}(\qu)A+|\qu|^2\Iop)(A\circ f)(\qu)=R_\qu(A)(A\circ f)(\qu)
		\end{eqnarray*}
	Since $(A\circ f)(\qu)$ is a continuous right slice-regular function on $U_0$ we have $\qu\in\rho_A^S(A\phi)$, and hence $\qu\not\in\ls(A\phi)$. Hence we have (a).\\
	(b)~~	Suppose that $\qu_0\not\in\sigma_{A}^S(\phi)$, that is  $\qu_0\in\rho_{A}^S(\phi)$. Then there exists an open neighborhood $U_0\subseteq\quat$, and a continuous right slice-regular function $f:U_0\longrightarrow\vr$, such that $R_\qu(A)f(\qu)=\phi$ for all $\qu\in U_0$. Hence,
	\begin{eqnarray*}
		BA\phi&=&BA(A^2-2\text{Re}(\qu)A+|\qu|^2\Iop)f(\qu)\\
		&=&(BA^3-2\text{Re}(\qu)BA^2+|\qu|^2BA)f(\qu)\\
		&=&(BA^2BA-2\text{Re}(\qu)(BA)^2+|\qu|^2BA)f(\qu)\\
		&=&((BA)^2-2\text{Re}(\qu)(BA)+|\qu|^2\Iop)(BA\circ f)(\qu)=R_\qu(BA)(BA\circ f)(\qu)
	\end{eqnarray*}
	Since $(BA\circ f)(\qu)$ is a continuous right slice-regular function on $U_0$ we have $\qu\in\rho_{BA}^S(BA\phi)$, and hence $\qu\not\in\sigma_{BA}^S(BA\phi)$ Hence we have (b).
\end{proof}	
%%%%%%%%%%%%%%%%%
\begin{proposition}\label{L-1}\cite{Ka}
	Let $A\in\B(\vr)$. Then,
	\begin{enumerate}
		\item [(a)] for every $\pu\in\quat\setminus\sus(A)$, there is an $r>0$ for which $\vr=\mathcal{X}_A(\quat\setminus B_\quat(\pu,r))$;
		\item[(b)] $\sus(A)=\bigcup\{\ls(\phi)~|~\phi\in\vr\}$;
		\item[(c)]if $A$ has SVEP and $\qu\in\sigma_{pS}(A)$, then $\ls(\phi)=\{\qu\}$ for each eigenvector $\phi$ of $A$ with respect to $\qu$;
		\item[(d)] $\sigma_S(A)=\sus(A)$ if $A$ has SVEP, and $\sigma_S(A)=\apo(A)$ if $A^\dagger$ has SVEP.
	\end{enumerate}	
\end{proposition}
%%%%%%%%%%%%%%%%%
\begin{proposition}\label{1.2.16}\cite{Ka}
	For every operator $A\in\B(\vr)$ and every set $F\subseteq\quat$, the following assertions hold:
	\begin{enumerate}
		\item [(a)] $V_A(F)$ is an $A$-hyperinvariant right linear subspace of $\vr$;
		\item[(b)] $R_\qu(A)V_A(F)\subseteq V_A(F)$ for all $\qu\in\quat\setminus F$;
		\item[(c)] if $Y$ is a $A$-invariant closed right linear subspace of $\vr$ with the property that $\sigma_S(A|Y)\subseteq F$, then $Y\subseteq V_A(F)$;
		\item[(d)] $V_A(F)=V_A(F\cap\sigma_S(A))$.
	\end{enumerate}
\end{proposition}
%%%%%%%%%
\begin{proposition}\label{EM}
	Let $A\in\B(\vr)$. If $A$ has SVEP, then $V_A(\emptyset)=\{0\}$.
\end{proposition}
\begin{proof}Suppose that $A$ has SVEP. Let $\phi\in V_A(\emptyset)$. Since, $\ls(\phi)=\emptyset$, $\lr(\phi)=\quat$, and hence there exists a continuous right slice-regular function $f:\quat\longrightarrow\vr$ such that $R_\pu f(\pu)=\phi$ for all $\pu\in\quat$. Since $f(\pu)=R_\pu(A)^{-1}\phi$ for all $\pu\in\rho_S(A)$ and, see theorem 3.1.5 in \cite{Jo}, $\displaystyle R_\pu(A)^{-1}=\sum_{n=0}^\infty A^n\sum_{k=0}^n \overline{\pu}^{(-k-1)}\pu^{-n+k-1}$, we have $\|R_\pu(A)^{-1}\|\longrightarrow 0$ as $|\pu|\rightarrow \infty$. Thus, $f$ is a bounded continuous right slice-regular function on $\quat$. Therefore,  by the vector-valued Liouville's theorem, $f$ is a constant. Since $R_\pu(A)^{-1}\phi\longrightarrow 0$ as $|\pu|\rightarrow 0$, we conclude that $f=0$ on $\quat$, and hence $\phi=0$. Therefore $V_A(\emptyset)=\{0\}$.
\end{proof}
%%%%%%%%%%%%%%%
The following proposition gathers some basic properties of global spectral subspaces.
\begin{proposition}\label{global}Let $F\subseteq\quat$ be a closed subset and $A\in\B(\vr)$.
	\begin{enumerate}
		\item [(a)]$\va(F)$ is a hyperinvariant subspace of $\vr$.
		\item[(b)] $\va(F)\subseteq V_A(F)$.
		\item[(c)] If $A$ has SVEP, then $V_A(F)=\va(F)$.
		\item[(d)] $\va(\emptyset)=\{0\}$.
		\item[(e)] $\va(\sa(A))=\vr$.
		\item[(f)] $\va(F)=\va(F\cap\sa(A))$.
	\end{enumerate}\
\end{proposition}
\begin{proof}
	(a)~~We have continuous right slice-regular function $f:\quat\setminus F\longrightarrow\vr$ given by $f(\qu)=0$ and $R_\qu(A)f(\qu)=0$ for all $\qu\in\quat\setminus F$. Thus $0\in\va(F)$. Let $\phi,\psi\in\va(F)$ and $\pu\in\quat$. Then there are continuous right slice-regular functions $f,g:\quat\setminus F\longrightarrow\vr$ such that $R_\qu(A)f(\qu)=\phi$ and $R_\qu(A)g(\qu)=\psi$ for all $\qu\in\quat\setminus F$. Hence, $f+g:\quat\setminus F\longrightarrow\vr$ is a continuous right slice-regular function such that $R_\qu(A)(f+g)(\qu)=\phi+\psi$ for all  $\qu\in\quat\setminus F$. Thus $\phi+\psi\in\va(F)$. Also $\pu f:\quat\setminus F\longrightarrow\vr$ given by $f(\qu)\pu$ is a continuous right slice-regular function such that $R_\qu(A)(\pu f)(\qu)=\phi\pu$ or all  $\qu\in\quat\setminus F$, and hence $\pu\phi\in\va(F)$. Therefore $\va(F)$ is a right linear subspace of $\vr$. Let $B\in\B(\vr)$ commutes with $A$. Let $\psi\in B\va(F)$, then there exists $\phi\in\va(F)$ such that $\psi=B\phi$. Since $\phi\in\va(F)$ there exits $f:\quat\setminus F\longrightarrow\vr$ such that $R_\qu(A)f(\qu)=\phi$ and for all $\qu\in\quat\setminus F$. Thus $R_\qu(A)(B\circ f(\qu)=B\phi=\psi$ and for all $\qu\in\quat\setminus F$ and $B\circ f$ is a continuous right slice-regular function. Thus $\psi\in\va(F)$, hence $\va(F)$ is hyperinvariant.\\
	(b)~~Let $\phi\in\va(F)$, then there are continuous right slice-regular functions $f:\quat\setminus F\longrightarrow\vr$ such that $R_\qu(A)f(\qu)=\phi$ for all $\qu\in\quat\setminus F$. Therefore $\qu\in\lr(\phi)$ for all $\qu\in\quat\setminus F$, and hence $\ls(\phi)\subseteq F$, which implies $\phi\in V_A(F)$.\\
	(c)~~Let $\phi\in V_A(F)$, then $\ls(\phi)\subseteq F$. Therefore, since $A$ has SVEP, $R_\qu(A)$ is invertible in $\quat\setminus F$. Define the  continuous right slice-regular functions $f:\quat\setminus F\longrightarrow\vr$ given by $f(\qu)=R_\qu(A)^{-1}\phi$. Then $R_\qu(A)f(\qu)=\phi$ for all $\qu\in\quat\setminus F$. Thus $\phi\in\va(F)$.\\
	(d)~~Let $\phi\in\va(\emptyset)$ then there is a continuous right slice-regular functions $f:\quat\longrightarrow\vr$ such that $R_\qu(A)f(\qu)=\phi$ for all $\qu\in\quat$. Thus $f(\qu)=R_\qu)(A)^{-1}\phi$ for all $\qu\in\rho_S(A)$. Thus, as in proposition \ref{EM}, $f(\qu)\longrightarrow 0$ as $|\qu|\longrightarrow\infty$. Hence, by the vector valued version of Liouville's theorem, $f=0$, and hence $\phi=0$.\\
	(e)~~Let $\phi\in\vr$, the we have $f:\rho_S(A)=\quat\setminus\sa(A)\longrightarrow\vr$, a continuous right slice-regular function, given by $f(\qu)=R_\qu(A)^{-1}\phi$ such that $R_\qu(A)f(\qu)=\phi$ for all $\phi\in\rho_S(A)$. Hence $\phi\in\va(\sa(A))$.\\
	(f)~~Let $\phi\in\va(F\cap\sa(A))$ then there is a continuous right slice-regular functions $f:\quat\setminus(F\cap\sa(A))\longrightarrow\vr$ such that $R_\qu(A)f(\qu)=\phi$ for all $\qu\in\quat\setminus(F\cap\sa(A))$. Hence,  $R_\qu(A)g(\qu)=\phi$ for all $\qu\in\quat\setminus(F)$, where $g$ is the restriction of $f$ to the set $\quat\setminus(F)$. Thus $\phi\in\va(F)$. For the opposite inclusion, let $\phi\in\va(F)$, then  then there is a continuous right slice-regular functions $f:\quat\setminus F\longrightarrow\vr$ such that $R_\qu(A)f(\qu)=\phi$ for all $\qu\in\quat\setminus F$. Define $h:\quat\setminus(F\cap\sa(A))\longrightarrow\vr$ by
	$$h(\qu)=\left\{\begin{array}{ccc}
	f(\qu)&\text{if}&\qu\in\quat\setminus F\\
	R_\qu(A)^{-1}\phi&\text{if}&\qu\in F\setminus\sa(A)\end{array}\right.$$
	Then $h$ is a continuous and right slice-regular function satisfying $R_\qu(A)h(\qu)=\phi$ for all $\qu\in\quat\setminus(F\cap\sa(A))$. Hence $\phi\in\va(F\cap\sa(A))$.
\end{proof}
%%%%%%%%%%%%%%%
The following result is a slight extension of the fact that the local spectral subspaces are hyperinvariant.
\begin{proposition}\label{1.2.17}
	Let $A\in\B(\vr)$, $B\in\B(\ur)$ and $R\in(\vr,\ur)$. Suppose that $BR=RA$, then $\sigma_B(R\phi)\subseteq\ls(\phi)$ for all $\phi\in\vr$, and $RV_A(F)\subseteq U_B(F)$ for all subsets $F$ of $\quat$.
\end{proposition}
\begin{proof}
	If $\qu\not\in\ls(\phi)$, then there is an open neighborhood $U\subseteq\quat$ of $\qu$ and a continuous right slice-regular function $f:U\longrightarrow\vr$ such that $R_\pu(A)f(\pu)=\phi$ for all $\pu\in U$. Now $R\circ f:U\longrightarrow\ur$ is a continuous right slice-regular function and
	\begin{eqnarray*}
		R_\pu(B)(R\circ f)(\pu)&=&(B^2R-2\text{Re}(\pu)BR+|\pu|^2R)f(\qu)\\
		&=&(BRA-2\text{Re}(\pu)RA+|\pu|^2R)f(\qu)\\
		&=&(RA^2-2\text{Re}(\pu)RA+|\pu|^2R)f(\qu)=R~R_\pu(A)f(\pu)\\
		&=&R\phi\quad\text{for all}~~\pu\in U.
	\end{eqnarray*}
Thus $\qu\in\rho_B(R\phi)$, and hence $\qu\not\in\sigma_B(R\phi)$. Now, let $\phi\in RV_A(F)$, then there exists $\psi\in V_A(F)$ such that $\phi=R\psi$. Since $\psi\in V_A(F)$, $\ls(\psi)\subseteq F$, thence by the previous result $\sigma_B(R\psi)=\sigma_B(\phi)\subseteq F$, and hence $\phi\in U_B(F)$.
\end{proof}
%%%%%%%%%%%%%%%%%%%%%%%%%
Let $A\in\B(\vr)$ has property ($\beta$), then, by remark \ref{DR2} and proposition \ref{L-1}, $\sa(A)=\sus(A)$. Similarly, if $A^\dagger$ has property ($\beta$) then $\sa(A)=\apo(A)$. However, for a decomposable operator the situation is particularly pleasant.
\begin{proposition}\label{1.3.3}
	Let $A\in\B(\vr)$ be decomposable. Then
	$$\sa(A)=\apo(A)=\sus(A)=\bigcup\{\ls(\phi)~|~\phi\in\vr\}.$$
\end{proposition}
\begin{proof}
	Since we know from theorem \ref{1.2.7} and remark \ref{DR2} part (a) that $A$ has SVEP. Therefore, from proposition \ref{L1} we conclude that $\sa(A)=\sus(A)=\bigcup\{\ls(\phi)~|~\phi\in\vr\}$. From proposition \ref{CP1}, we have $\apo(A)\subseteq\sa(A)$. It only remains to prove that $\sa(A)\subseteq\apo(A)$. Given an arbitrary $\qu\in\sa(A)$ and any $\epsilon>0$, let $U\subseteq\quat$ be an open set for which $\qu\not\in U$ and $U\cup B_\quat(\qu,\epsilon)=\quat$. Therefore, by the definition of decomposability, there exists $A$-invariant closed right linear subspaces $Y,Z\subseteq\vr$ for which $\vr=Y+Z$, $\sa(A|Y)\subseteq U$ and $\sa(A|Z)\subseteq B_\quat(\qu,\epsilon).$ Note that $Z$ is non-trivial, since otherwise $\vr=Y$, and hence $\sa(A)=\sa(A|Y)\subseteq U$, which is a contradiction to the fact that $\qu\not\in U$. Since $Z$ is non-trivial,  by theorem \ref{T1}, $\partial\sa(A|Z)\subseteq\apo(A|Z)\subseteq\apo(A)$. Thus, it follows from $\sa(A|Z)\subseteq B_\quat(\qu,\epsilon)$ that  $B_\quat(\qu,\epsilon)\cap\apo(A)\not=\emptyset$ for all $\epsilon>0$. Since, by theorem \ref{T1}, $\apo(A)$ is a closed set, and hence $\qu\in\apo(A)$, which completes the proof.
\end{proof}
Proposition \ref{1.3.3} may, of course, be used to show that a certain operator fails to be decomposable. See the example below.
%%%%%%%%%%%%%%%%%%%%%%%%%%%%%%%%%%
\begin{definition}\label{1.2.18}
	An operator $A\in\B(\vr)$ has Dunford's property (C) if the local spectral subspace $V_A(F)$ is closed for every closed set $F\subseteq\quat$.
	\end{definition}
%%%%%%%%%%%%%%%%%%%%%%%
\begin{proposition}\label{1.2.19}
	If $A\in\B(\vr)$ has property ($\beta$), then $A$ has property (C).
\end{proposition}
\begin{proof}
	Let $F\subseteq\quat$ be an arbitrary closed set. The space $\vr$ may be identified with the space of constant functions in $H(\quat\setminus F,\vr)$ by defining, for each $\phi\in\vr$, $f:\quat\setminus F\longrightarrow\vr$ with $f(\qu)=\phi$ for all $\qu\in\quat\setminus F$. In this identification, the norm topology of $\vr$ coincides with the topology induced by the metric on  $H(\quat\setminus F,\vr)$. By remark \ref{DR2}, property ($\beta$) implies SVEP. Let $\phi\in V_A(F)$. By SVEP, $R_\qu(A)^{-1}$ exists for all $\qu\in\quat\setminus F$, and thus define $f:\quat\setminus F\longrightarrow\vr$ by $f(\qu)=R_\qu(A)^{-1}\phi$, then $f$ is a continuous right slice-regular function on $\quat\setminus F$. Thus, $R_\qu(A)f(\qu)=\phi$ for all $\qu\in\quat\setminus F$. That is, $(A_{\quat\setminus F}f)(\qu)=\phi$ for all $\qu\in\quat\setminus F$. Hence $\phi\in A_{\quat\setminus F}H(A_{\quat\setminus F},\vr)$, and therefore
	$$V_A(F)=\vr\cap A_{\quat\setminus F}H({\quat\setminus F},\vr).$$
	Hence, by proposition \ref{1.2.6}, $V_A(F)$ is closed.
\end{proof}
%%%%%%%%%%%%%%
For $A\in\B(\vr)$,
$$\kappa(A)=\inf\{\|A\phi\|~~|~~\phi\in\vr~~\text{with}~~\|\phi\|=1\}$$  denotes the lower bound of $A$, then $\kappa(A^m)\kappa(A^n)\leq\kappa(A^{m+n})$ for all $m,n\in\N$. The following limit exists (see \cite{Ka} for details)
	$$i(A)=\lim_{n\rightarrow\infty}\kappa(A^n)^{1/n}=\sup_{n\in\N}\kappa(A^n)^{1/n}.$$
	\begin{proposition}\cite{Ka}\label{1.6.2} Every operator $A\in\B(\vr)$ has the following property.
		 $\apo(A)$ is contained in the spherical annulus $\{\qu\in\quat~~|~~i(A)\leq|\qu|\leq r_S(A)\}$, where $\displaystyle r_S(A)=\lim_{n\rightarrow\infty}\|A^n\|^{1/n}=\inf_{n\in\N}\|A^n\|^{1/n}$ is the spectral radius.
		\end{proposition}
	\begin{example}\label{page37}
			 Consider the right quaternionic Hilbert space
		$$l^2(\N)=\left\{x:\N\longrightarrow\quat~~|~~\sum_{i=1}^\infty |x_i|^2<\infty\right\}$$
		 Consider the unilateral right shift $B$ on $l^2(\N)$ given by
		$$B(x_1,x_2,x_3,\cdots)=(0,x_1,x_2,\cdots)\quad\text{for all}~~~(x_1,x_2,x_3,\cdots)\in l^2(\N).$$
		Its adjoint is $B^\dagger(x_1,x_2,x_3,\cdots)=(x_2,x_3,\cdots)$. Also, since $\|Bx\|=\|x\|$ for all $x\in l^2(\N)$, we have $\|B\|=1$.  Consider the right eigenvalue problem of $B^\dagger$. That is, if $B^\dagger(x_1,x_2,x_3,\cdots)=(x_1,x_2,x_3,\cdots)\qu$, then $(x_2,x_3,\cdots)=(x_1,x_2,x_3,\cdots)\qu$. Hence $x_{n+1}=x_n\qu$ for all $n\geq 1$, and therefore $x_n=x_1\qu^{n-1}$ for all $n\geq 1$. That is, $(x_1,x_2,x_3,\cdots)=x_1(1,\qu,\qu^2,\cdots)$. Thus, by proposition \ref{su1}, $\qu\in\sigma_{pS}(B^\dagger)\subseteq\sigma_S(B^\dagger)=\sa(B)$, if $|\qu|<1$. Therefore, since $\|B\|=1$, the S-spectral radius $\displaystyle r_S(B)=\lim_{n\rightarrow\infty}\|B^n\|^{1/n}=1$, and $\sa(B)$ is compact, we have $\sa(B)=\nabla_\quat(0,1)$ the closed quaternionic unit ball. Since $\|Bx\|=\|x\|$ for all $x\in l^2(\N)$, $i(B)=1$, thus by proposition \ref{1.6.2}, $\apo(B)\subseteq\partial\nabla_\quat(0,1)$, the boundary of $\nabla_\quat(0,1)$. Thus $\sa(B)\not=\apo(B)$, hence by proposition \ref{1.3.3}, the operator $B$ is not decomposable.
		\end{example}
	%%%%%%%%%%%%%%%%%%%%%%%%%%%%%%%%%%%%%%%%%%%%%%%%%
\section{Acknowledments}
K. Thirulogasanthar would like to thank the FRQNT, Fonds de la Recherche  Nature et  Technologies (Quebec, Canada) for partial financial support under the grant number 2017-CO-201915. Part of this work was done while he was visiting the University of Jaffna to which he expresses his thanks for the hospitality.


\begin{thebibliography}{XXXX}
\bibitem{Ad} Adler, S.L., {\em Quaternionic Quantum Mechanics and Quantum Fields}, Oxford University Press, New York, 1995.

\bibitem{Ai} Aiena, P., {\em Fredholm and local spectral theory, with applications to multipliers}, Kluwer Academic Publishers, Dordrecht, 2004.

\bibitem{AC} Alpay, D., Colombo, F., Kimsey, D.P., {\em The spectral theorem for quaternionic unbounded normal operators based on the $S$-spectrum}, J. Math. Phys. {\bf 57} (2016), 023503.

\bibitem{Al} Alpay, D., Colombo, F., Sabadini, I., {\em Slice hyperholomorphic Schur analysis}, Springer International Publishing (2016).


\bibitem{Am} Buchmann, A., {\em A brief history of quaternions and the theory of holomorphic functions of quaternionic variables}, arXiv:1111.6088v1[Math.HO]

\bibitem{Fab} Colombo, F., Sabadini, I., \textit{On Some Properties of the Quaternionic Functional Calculus}, J. Geom. Anal., {\bf 19} (2009), 601-627.

\bibitem{Fab1} Colombo, F., Sabadini, I., \textit{On the  Formulations of the Quaternionic Functional Calculus}, J. Geom. Phys., {\bf 60} (2010), 1490-1508.

\bibitem{Fab2} Colombo, F., Gentili, G., Sabadini, I., Struppa, D.C., {\em Non commutative functional calculus: Bounded operators}, Complex Analysis and Operator Theory, {\bf 4} (2010), 821-843.
%\bibitem{cspectrum} Colombo, F., Sabadini, I., \textit{On some notions of spectra for quaternionic operators and for n-tuples of operators}, (con F. Colombo), C. R. Acad. Sci. Paris, Ser. I, {\bf 350} (2012), 399--402.


\bibitem{NFC} Colombo, F., Sabadini, I., Struppa, D.C., {\em Noncommutative Functional Calculus}, Birkh\"auser Basel, 2011.

\bibitem{Jo} Colombo, F., Gantner, J., Kimsey, D.P., {\em Spectral theory on the S-spectrum for quaternionic operators},

\bibitem{ESR} Colombo, F., Sabadini, I., Struppa, D.C., {\em Entire slice regular functions}, Springer, Switzerland, 2016.

\bibitem{ghimorper} Ghiloni, R., Moretti, W. and Perotti, A., {\em Continuous slice functional calculus in quaternionic Hilbert spaces\/,} Rev. Math. Phys. {\bf 25} (2013), 1350006.

\bibitem{Gra1} Gentili, G. and Struppa, D.C., {\em A new theory of regular functions of a quaternionic variable}, Adv. Math. {\bf 216} (2007), 279-301.

\bibitem{Gra2} Gentili, G. and Stoppato, C., {\em Power series and analyticity over the quaternions}, Math. Ann. {\bf 352} (2012), 113-131.

\bibitem{GSS}Gentili, G., Stoppato, C., Struppa, D.C., {\em Regular function for a quaternionic variable: Springer Monographs in Mathematics}, Springer, Berlin, 2013.


\bibitem{La} Laursen, K.B., Neumann, M.M., {\em An introduction to local spectral theory}, Oxford University Press, Oxford, 2000.

\bibitem{BT} Muraleetharan, B., Thirulogasanthar, K., {\em Deficiency Indices of Some Classes of Unbounded $\quat$-Operators},  Complex Anal. Oper. Theory (2017), 1-29. https://doi.org/10.1007/s11785-017-0702-4.

\bibitem{Mu} Muraleetharan, B, Thirulogasanthar, K., {\em Coherent state quantization of quaternions}, J. Math. Phys., {\bf 56} (2015), 083510.

\bibitem{Fr}  Muraleetharan, B, Thirulogasanthar, K., {\em Fredholm operators and essential S-spectrum in the quaternionic setting}, J. Math. Phys., {\bf 59} (2018), 103506.

\bibitem{Ka}   Thirulogasanthar, K.,  Muraleetharan, B., {\em Kato S-spectrum in the quaternionic setting}, 	arXiv:1904.02977.

%\bibitem{Mus}  Muscat, J., \textit{Functional Analysis: An Introduction to Metric Spaces, Hilbert Spaces, and Banach Algebras}, Springer, 2014.

%\bibitem{Rud} Rudin, W., \textit{Functional Analysis}, International Series in Pure and Applied Mathematics, McGraw-Hill Inc., New York, 1991.

%\bibitem{S}Stoppato, C., {\em A new series expansion for slice regular functions}, Adv. Math. {\bf 231} (2012), 1401-1416.

\bibitem{Vis} Viswanath, K., {\em Normal operators on quaternionic Hilbert spaces}, Trans. Amer. Math. Soc. {\bf 162} (1971), 337-350.

%\bibitem{VonN} von Neumann J., {\em Allgemeine Eigenwerttheorie Hermitescher Funktionaloperatoren}, Math. Ann. {\bf 102} (1929-1930), 49-131.

%\bibitem{Wil} Willians, V., {\em Closed Fredholm and semi-Fredholm operators, essential spectra and perturbations}, J. Funct. Anal. {\bf 20} (1975), 1-25.

\end{thebibliography}
\end{document}